\documentclass[noinfoline]{imsart}
\usepackage{amsthm,amsmath,amssymb,amsfonts,delarray}
\usepackage{relsize}
\usepackage{lscape}
\usepackage[OT1]{fontenc}
\usepackage[dvipdfmx,colorlinks,citecolor=blue,urlcolor=blue]{hyperref}
\usepackage[numbers]{natbib}
\usepackage{subfigure}
\usepackage[dvipdfm]{graphicx}
\newtheorem{thm}{Theorem}[section]
\newtheorem{lem}[thm]{Lemma}

\theoremstyle{definition}
\newtheorem{rem}[thm]{Remark}
\begin{document}

\begin{frontmatter}
\title{Asymptotically Minimax Prediction \\ in Infinite Sequence Models}
\runtitle{Prediction in Infinite Sequence Models}
\begin{aug}
	\author{Keisuke Yano \corref{}\ead[label=e1]{yano@mist.i.u-tokyo.ac.jp}}\\
	\affiliation{The University of Tokyo}
	\address{Department of Mathematical Informatics\\
	Graduate School of Information Science and Technology\\
	The University of Tokyo\\
	7-3-1 Hongo, Bunkyo-ku, Tokyo 113-8656, Japan}
	\printead{e1}\\
\and
\author{Fumiyasu Komaki \ead[label=e2]{komaki@mist.i.u-tokyo.ac.jp}}\\
\affiliation{The University of Tokyo}
\address{Department of Mathematical Informatics\\
	Graduate School of Information Science and Technology\\
	The University of Tokyo\\
	7-3-1 Hongo, Bunkyo-ku, Tokyo 113-8656, Japan}
\printead{e2}\\
\address{RIKEN Brain Science Institute\\
2-1 Hirosawa, Wako City, Saitama 351-0198, Japan}
\runauthor{K. Yano and F. Komaki}
\end{aug}

\begin{abstract}
We study asymptotically minimax predictive distributions in infinite sequence models.
First,
we discuss the connection between prediction in an infinite sequence model and prediction in a function model.
Second,
we construct an asymptotically minimax predictive distribution for the setting in which the parameter space is a known ellipsoid.
We show that the Bayesian predictive distribution based on the Gaussian prior distribution is asymptotically minimax in the ellipsoid.
Third, we construct an asymptotically minimax predictive distribution for any Sobolev ellipsoid.
We show that the Bayesian predictive distribution based on the product of Stein's priors is asymptotically minimax for any Sobolev ellipsoid.
Finally, we present an efficient sampling method from the proposed Bayesian predictive distribution.
\end{abstract}


\end{frontmatter}

\section{Introduction}

We consider prediction in an infinite sequence model.
The current observation is a random sequence $X=(X_{1},X_{2},\ldots)$ given by
\begin{align}
X_{i}=\theta_{i}+\varepsilon W_{i} \text{ for $i\in\mathbb{N}$},
\label{current_insequence}
\end{align}
where $\theta=(\theta_{1},\theta_{2},\ldots)$ is an unknown sequence in $l_{2}:=\{\theta:\sum_{i=1}^{\infty}\theta_{i}^{2}<\infty\}$
and
$W=(W_{1},W_{2},\ldots)$ is a random sequence distributed according to $\mathop{\otimes}_{i=1}^{\infty}\mathcal{N}(0,1)$ 
on $(\mathbb{R}^{\infty},\mathcal{R}^{\infty})$.
Here $\mathcal{R}^{\infty}$ is a product $\sigma$-field of the Borel $\sigma$-field $\mathcal{R}$ on the Euclidean space $\mathbb{R}$.
Based on the current observation $X$,
we estimate the distribution of a future observation $Y=(Y_{1},Y_{2},\ldots)$ given by
\begin{align}
Y_{i}=\theta_{i}+\tilde{\varepsilon}\widetilde{W}_{i} \text{ for $i\in\mathbb{N}$},
\label{future_insequence}
\end{align}
where $\widetilde{W}=(\widetilde{W}_{1},\widetilde{W}_{2},\ldots)$ is distributed according to $\mathop{\otimes}_{i=1}^{\infty}\mathcal{N}(0,1)$.
We denote the true distribution of $X$ with $\theta$ by $P_{\theta}$
and 
the true distribution of $Y$ with $\theta$ by $Q_{\theta}$.
For simplicity,
we assume that $W$ and $\widetilde{W}$ are independent.

Prediction in an infinite sequence model is shown to be equivalent to the following prediction in a function model. 
Consider that we observe a random function $X(\cdot)$ given by
\begin{align}
	X(\cdot)=F(\cdot)+\varepsilon W(\cdot) \text{ in } L_{2}[0,1],
\label{current_infunction}
\end{align}
where 
$L_{2}[0,1]$ is the $L_{2}$-space on $[0,1]$ with the Lebesgue measure,
$F(\cdot):[0,1]\to\mathbb{R}$ is an unknown absolutely continuous function of which the derivative is in $L_{2}[0,1]$,
$\varepsilon$ is a known constant,
and
$W(\cdot)$ follows the standard Wiener measure on $L_{2}[0,1]$.
Based on the current observation $X(\cdot)$,
we estimate the distribution of a random function $Y(\cdot)$ given by
\begin{align}
	Y(\cdot)=F(\cdot)+\tilde{\varepsilon} \widetilde{W}(\cdot) \text{ in } L_{2}[0,1],
	\label{future_infunction}
\end{align}
where
$\tilde{\varepsilon}$ is a known constant,
and
$\widetilde{W}(\cdot)$ follows the standard Wiener measure on $L_{2}[0,1]$.
The details are provided in Section \ref{Sec:Equivalence}.
\citet{XuandLiang(2010)} established the connection between 
prediction of a function on equispaced grids
and prediction in a high-dimensional sequence model,
using the asymptotics in which the dimension of the parameter grows to infinity according to the growth of the grid size.
Our study is motivated by \citet{XuandLiang(2010)}
and is its generalization to the settings in which the parameter $\theta$ is infinite-dimensional.

Using the above equivalence,
we discuss the performance of a predictive distribution $\widehat{Q}(\cdot;\cdot)$ of $Y$ based on $X$ in an infinite sequence model.
Let $\mathcal{A}$ be the whole set of probability measures on $(\mathbb{R}^{\infty},\mathcal{R}^{\infty})$
and
let $\mathcal{D}$ be the decision space $\{\widehat{Q}:\mathbb{R}^{\infty}\to\mathcal{A}\}$.
We use
the Kullback--Leibler loss
as a loss function:
for all $Q\in\mathcal{A}$ and all $\theta\in l_{2}$,
if $Q_{\theta}$ is absolutely continuous with respect to $Q$, then
\begin{align*}
	l(\theta,Q):=\int \log \frac{\mathrm{d}Q_{\theta}}{\mathrm{d}Q}(y)\mathrm{d}Q_{\theta}(y),
\end{align*}
and otherwise $l(\theta,Q)=\infty$.
The risk of a predictive distribution $\widehat{Q}(\cdot;\cdot)\in\mathcal{D}$
in the case that the true distributions of $X$ and $Y$ are $P_{\theta}$ and $Q_{\theta}$,
respectively,
is denoted by
\begin{align*}
	R(\theta,\widehat{Q}):= \int l(\theta,\widehat{Q}(\cdot;x)) \mathrm{d}P_{\theta}(x).
\end{align*}

We construct an asymptotically minimax predictive distribution $\widehat{Q}^{*}\in\mathcal{D}$ that satisfies
\begin{align*}
	\lim_{\varepsilon\to 0}\left[\mathop{\sup}_{\theta\in\Theta(a,B)}R(\theta,\widehat{Q}^{*})
		\bigg{/}\mathop{\inf}_{\widehat{Q}\in\mathcal{D}}\mathop{\sup}_{\theta\in\Theta(a,B)}R(\theta,\widehat{Q})
	\right]=1,
\end{align*}
where $\Theta(a,B):=\{\theta\in l_{2}:\sum_{i=1}^{\infty} a_{i}^{2}\theta_{i}^{2}\leq B\}$
with a known non-zero and non-decreasing divergent sequence $a=(a_{1},a_{2},\ldots)$ and
with a known constant $B$.
Note that for any $\varepsilon>0$, the minimax risk is bounded above by $(1/2\tilde{\varepsilon}^{2})(B/a^{2}_{1})<\infty$.
Further, note that using the above equivalence between the infinite sequence model and the function model,
the parameter restriction in the infinite sequence model that $\theta\in\Theta(a,B)$ 
corresponds to the restriction that
the corresponding parameter in the function model is smooth;
$B$ represents the volume of the parameter space,
and 
the growth rate of $a$ represents the smoothness of the functions.

The constructed predictive distribution is the Bayesian predictive distribution based on the Gaussian distribution.
For a prior distribution $\Pi$ of $\theta$,
the Bayesian predictive distribution $Q_{\Pi}$ based on $\Pi$ is
obtained by averaging $Q_{\theta}$ with respect to the posterior distribution based on $\Pi$.
Our construction is a generalization of the result in \citet{XuandLiang(2010)} to infinite-dimensional settings.
The details are provided in Section \ref{Sec:AsymptoticallyMinimaxTheorem}.

Further,
we discuss adaptivity to the sequence $a$ and $B$.
In applications,
since we do not know the true values of $a$ and $B$,
it is desirable to construct a predictive distribution without using $a$ and $B$ that is asymptotically minimax in any ellipsoid in the class.
Such a predictive distribution is called an asymptotically minimax adaptive predictive distribution in the class.
In the present paper,
we focus on an asymptotically minimax adaptive predictive distribution
in the simplified Sobolev class $\{\Theta_{\mathrm{Sobolev}}(\alpha,B):\alpha>0,B>0\}$,
where $\Theta_{\mathrm{Sobolev}}(\alpha,B):=\{\theta\in l_{2}:\sum_{i\in\mathbb{N}}i^{2\alpha}\theta^{2}_{i} \leq B \}$.

Our construction of the asymptotically minimax adaptive predictive distribution 
is based on Stein's prior and the division of the parameter into blocks.
The proof of the adaptivity relies on a new oracle inequality related to the Bayesian predictive distribution based on Stein's prior;
see Subsection \ref{Subsec:Oracle}.
Stein's prior on $\mathbb{R}^{n}$ is an improper prior
whose density is $\left(\sum_{i=1}^{n}\theta_{i}^{2}\right)^{(2-n)/n}$.
It is known that
the Bayesian predictive distribution based on that prior has a smaller Kullback--Leibler risk than that based on the uniform prior
in the finite dimensional Gaussian settings;
see \citet{Komaki(2001)} and \citet{GeorgeLiangXu(2006)}.
The division of the parameter into blocks is widely used
for the construction of the asymptotically minimax adaptive estimator;
see \citet{EfromovichandPinsker(1984)}, \citet{Caietal(2000)}, and \citet{CavalierandTsybakov(2001)}.
The details are provided in Section \ref{Sec:Adaptivity}.

The remainder of the paper is organized as follows.
In Section \ref{Sec:Numerical}, we provide an efficient sampling method for the proposed asymptotically minimax adaptive distribution
and provide numerical experiments with a fixed $\varepsilon$.
In Section \ref{Sec:Conclusion}, we conclude the paper.

\section{Equivalence between predictions in infinite sequence models and predictions in function models}
\label{Sec:Equivalence}

In this section,
we provide an equivalence between prediction in a function model and prediction in an infinite sequence model.
The proof consists of the two steps.
First, we provide a connection between predictions in a function model and predictions in the submodel of an infinite sequence model.
Second, we extend predictions in the submodel to predictions in the infinite sequence model.

The detailed description of prediction in a function model is as follows.
Let $\mathcal{H}_{\mathrm{F}}:=\{F(\cdot)\in L_{2}[0,1]:F(0)=0,\dot{F}(\cdot)\in L_{2}[0,1]\}$,
where $\dot{F}$ denotes the derivative of $F$.
Let $\langle\cdot,\cdot\rangle_{L_{2}}$ be the inner product of $L_{2}[0,1]$.
Let $\mathcal{A}_{\mathrm{F}}$ be the whole set of probability distributions on $(L_{2}[0,1],\mathcal{B}_{\mathrm{F}})$,
where $\mathcal{B}_{\mathrm{F}}$ is the Borel $\sigma$-field of $H_{\mathrm{F}}$.
$l_{\mathrm{F}}(F,Q)$ denotes
the Kullback--Leibler loss of $Q\in \mathcal{A}_{\mathrm{F}}$ in the setting that the true parameter function is $F(\cdot)$.

Let $C:L_{2}[0,1]\to L_{2}[0,1]$ be the covariance operator of $W$:
for any $x(\cdot)\in L_{2}[0,1]$,
$C(x(\cdot))(\cdot)=\int_{0}^{1}(\cdot \wedge t)x(t)\mathrm{d}t$.
By Mercer's theorem,
there exists a non-negative monotone decreasing sequence $\{\lambda_{i}\}_{i=1}^{\infty}$
and an orthonormal basis $\{e_{i}(\cdot)\}_{i=1}^{\infty}$ in $L_{2}[0,1]$
such that
\begin{align*}
	C(x(\cdot))(\cdot)=\sum_{i=1}^{\infty}\lambda_{i}\langle x(\cdot) ,e_{i}(\cdot) \rangle_{L_{2}}e_{i}(\cdot) \text{ in } L_{2}[0,1].
\end{align*}
Explicitly,
$\lambda_{i}$ is $1/\{\pi(i-1/2)\}^{2}$
and
$e_{i}(\cdot)$ is $\sqrt{2}\sin((i-1/2)\pi \cdot)$
for $i\in\mathbb{N}$.

The detailed description of prediction in the sub-model of an infinite sequence model is as follows.
Let $S_{\mathrm{D}}$ be $\{x\in \mathbb{R}^{\infty}: \sum_{i=1}^{\infty}x_{i}\sqrt{\lambda_{i}}e_{i}(\cdot) \text{ converges in $L_{2}[0,1]$.} \}$.
Note that $S_{\mathrm{D}}$ is a measurable set with respect to $(\mathbb{R}^{\infty},\mathcal{R}^{\infty})$,
because
$g(x):=\sum_{i}^{\infty}x_{i}\sqrt{\lambda_{i}}e_{i}(\cdot)$ is the pointwise $L_{2}[0,1]$-limit of
$g_{n}(x):=\sum_{i=1}^{n}x_{i}\sqrt{\lambda_{i}}e_{i}(\cdot)$
and we use Theorem 4.2.2. in \citet{Dudley_RAP}.
Let $\mathcal{A}_{\mathrm{D}}$ be the whole set of probability distributions on $(S_{\mathrm{D}},\mathcal{B}_{\mathrm{D}})$,
where $\mathcal{B}_{\mathrm{D}}$ is the relative $\sigma$-field of $\mathcal{R}^{\infty}$.

The following theorem states that the Kullback--Leibler loss in the function model is equivalent to that in the submodel of the infinite sequence model.

\begin{thm}
	\label{identification_FD}
	For every $Q \in \mathcal{A}_{\mathrm{F}}$ and every $F \in \mathcal{H}_{\mathrm{F}}$,
	there exist $\widetilde{Q}\in \mathcal{A}_{\mathrm{D}} $ and $\theta\in l_{2}$
	such that
	\begin{align*}
		l_{\mathrm{F}}(F ,Q)=l(\theta,\widetilde{Q}).
	\end{align*}
	Conversely,
	for every $\widetilde{Q} \in \mathcal{A}_{\mathrm{D}} $ and every $\theta\in l_{2}$,
	there exist $Q \in\mathcal{A}_{\mathrm{F}}$ and $F\in \mathcal{H}_{\mathrm{F}}$
	such that
	\begin{align*}
		l(\theta,\widetilde{Q})=l_{\mathrm{F}}(F,Q).
	\end{align*}
\end{thm}
\begin{proof}
	We construct pairs of a measurable one-to-one map $\Phi:L_{2}[0,1]\to S_{\mathrm{D}}$ and a measurable one-to-one map $\Psi:S_{\mathrm{D}}\to L_{2}[0,1]$.

	Let $\Phi$ be defined by
	\begin{align*}
		\Phi (x(\cdot)):= \begin{pmatrix} \langle x(\cdot),\lambda_{1}^{-1/2}e_{1}(\cdot) \rangle_{L_{2}} \\
			\langle x(\cdot),\lambda_{2}^{-1/2}e_{2}(\cdot) \rangle_{L_{2}}  \\
	\cdots \end{pmatrix}.
	\end{align*}
	$\Phi$ is well-defined as a map from $L_{2}[0,1]$ to $S_{\mathrm{D}}$
	because for $x(\cdot)$ and $y(\cdot)$ in $L_{2}[0,1]$ such that $x(\cdot)=y(\cdot)$,
	we have 
	$\langle x(\cdot),\lambda_{i}^{-1/2}e_{i}(\cdot) \rangle_{L_{2}}=\langle y(\cdot), \lambda_{i}^{-1/2}e_{i}(\cdot) \rangle_{L_{2}}$,
	and 
	because
	for $x(\cdot)\in L_{2}[0,1]$,
	we have $\sum_{i=1}^{\infty}\langle x(\cdot),\lambda_{i}^{-1/2} e_{i}(\cdot)\rangle_{L_{2}} \lambda_{i}^{1/2}e_{i}(\cdot)\in L_{2}[0,1]$.

	We show that $\Phi$ is one-to-one, onto, and measurable.
	$\Phi$ is one-to-one
	because 
	if $\Phi(x(\cdot))=\Phi(y(\cdot))$,
	then
	we have $\langle x(\cdot),e_{i}(\cdot)\rangle_{L_{2}}=\langle y(\cdot),e_{i}(\cdot) \rangle_{L_{2}}$ for all $i\in \mathbb{N}$.
	$\Phi$ is onto 
	because
	if $x\in S_{\mathrm{D}}$,
	$x(\cdot):=\sum_{i=1}^{\infty}x_{i}\sqrt{\lambda_{i}}e_{i}(\cdot)$ satisfies that $\Phi(x(\cdot))=x$.
	$\Phi$ is measurable
	because
	$\Phi$ is continuous with respect to the norm $||\cdot||_{L_{2}}$ of $L_{2}[0,1]$ and $\rho$,
	and
	because $\mathcal{R}^{\infty}$ is equal to the Borel $\sigma$-field with respect to $\rho(x,y):=\sum_{i=1}^{\infty}2^{-i}|x_{i}-y_{i}|\wedge 1$.
	$\Phi$ is continuous, because we have
	\begin{align*}
	\rho(\Phi(x(\cdot)),\Phi(y(\cdot)))=\sum_{i=1}^{\infty}(\lambda^{-1/2}_{i}/2^{i})
	|\langle x(\cdot),e_{i}(\cdot) \rangle_{L_{2}}-\langle y(\cdot),e_{i}(\cdot) \rangle_{L_{2}} |\wedge 1.
	\end{align*}
	Further,
	the restriction of $\Phi$ to $\mathcal{H}_{\mathrm{F}}$ is a one-to-one and onto map from $\mathcal{H}_{\mathrm{F}}$ to $l_{2}$.

	Let $\Psi: S_{\mathrm{D}}\to L_{2}[0,1]$ be defined by $\Psi(x):=\sum_{i=1}^{\infty}x_{i}\sqrt{\lambda_{i}}e_{i}(\cdot)$.
	$\Psi$ is the inverse of $\Phi$.
	Thus, $\Psi$ is one-to-one, onto, and measurable.

	Since the Kullback--Leibler divergence is unchanged under a measurable one-to-one mapping,
	the proof is completed.
\end{proof}

\begin{rem}
	\citet{Mandelbaum(1984)} constructed the connection between estimation in an infinite sequence model and estimation in a function model.
	Our connection is its extension to prediction.
	In fact, the map $\Phi$ is used in \citet{Mandelbaum(1984)}.
\end{rem}

The following theorem justifies focusing on prediction in $(\mathbb{R}^{\infty},\mathcal{R}^{\infty})$ 
instead of prediction in $(S_{\mathrm{D}},\mathcal{B}_{\mathrm{D}})$.

\begin{thm}
	For every $\theta\in l_{2}$ and $Q\in \mathcal{A}$,
	there exists $\widetilde{Q}\in \mathcal{A}_{\mathrm{D}}$
	such that
	\begin{align*}
		l(\theta,\widetilde{Q})\leq l(\theta,Q).
	\end{align*}
	In particular,
	for any subset $\Theta$ of $l_{2}$,
	\begin{align*}
		\inf_{\widehat{Q}\in\mathcal{D}}\sup_{\theta\in\Theta}R(\theta,\widehat{Q})
		=\inf_{\widehat{Q}\in\mathcal{D}_{\mathrm{D}}}\sup_{\theta\in\Theta}R(\theta,\widehat{Q}),
	\end{align*}
	where $\mathcal{D}_{\mathrm{D}}:=\{\widehat{Q}:\mathbb{R}^{\infty}\to\mathcal{A}_{\mathrm{D}}\}$.
\end{thm}
\begin{proof}
	Note that $Q_{\theta}(S_{\mathrm{D}})=1$ by the Karhunen--Lo\`eve theorem.
	For $Q\in\mathcal{A}$ such that $Q(S_{\mathrm{D}})=0$,
	$l(\theta,Q)=\infty$ and then 
	for any $\widetilde{Q}\in\mathcal{A}_{\mathrm{D}}$,
	$l(\theta,\tilde{Q})<l(\theta,Q)$.
	For $Q\in\mathcal{A}$ such that $Q(S_{\mathrm{D}})>0$,
	\begin{align*}
		l(\theta,Q)=l(\theta,\widetilde{Q})-\log Q(S_{\mathrm{D}})\geq l(\theta,\widetilde{Q}),
	\end{align*}
	where $\widetilde{Q}$ is the restriction of $Q$ to $S_{\mathrm{D}}$.
\end{proof}

\section{Asymptotically minimax predictive distribution}
\label{Sec:AsymptoticallyMinimaxTheorem}

In this section,
we construct an asymptotically minimax predictive distribution 
for the setting in which the parameter space is an ellipsoid $\Theta(a,B)=\{\theta\in l_{2}:\sum_{i=1}^{\infty}a_{i}^{2}\theta_{i}^{2}\leq B\}$
with a known sequence $a=(a_{1},a_{2},\ldots)$ and with a known constant $B$.
Further,
we provide the asymptotically minimax predictive distributions in two well-known ellipsoids;
a Sobolev and an exponential ellipsoids.

\subsection{Principal theorem of Section \ref{Sec:AsymptoticallyMinimaxTheorem}}

We construct an asymptotically minimax predictive distribution in Theorem \ref{AsymptoticallyMinimaxTheorem}.

We introduce notations used in the principal theorem.
For an infinite sequence $\tau=(\tau_{1},\tau_{2},\ldots)$,
let $\mathrm{G}_{\tau}$ be $\mathop{\otimes}_{i=1}^{\infty}\mathcal{N}(0,\tau_{i}^{2})$ with variance $\tau^{2}=(\tau_{1}^{2},\tau_{2}^{2},\ldots)$.
Then,
the posterior distribution $\mathrm{G}_{\tau}(\cdot|X)$ based on $\mathrm{G}_{\tau}$ is
\begin{eqnarray}
\mathrm{G}_{\tau}(\cdot|X=x)
&=&
\mathop{\otimes}_{i=1}^{\infty}
\mathcal{N}\left(\frac{1/\varepsilon^{2}}{1/\varepsilon^{2}+1/\tau_{i}^{2}}x_{i},\frac{1}{1/\varepsilon^{2}+1/\tau^{2}_{i}}\right)
\text{ $P_{\theta}$-a.s.}.
\label{Gaussian_Posterior}
\end{eqnarray}
The Bayesian predictive distribution $Q_{\mathrm{G}_{\tau}}(\cdot|X)$ based on $\mathrm{G}_{\tau}$ is
\begin{eqnarray}
	Q_{\mathrm{G}_{\tau}}(\cdot|X=x)
&=&
\mathop{\otimes}_{i=1}^{\infty}
\mathcal{N}
\left(
\frac{1/\varepsilon^{2}}{1/\varepsilon^{2}+1/\tau^{2}_{i}}x_{i},
\frac{1}{1/\varepsilon^{2}+1/\tau^{2}_{i}}+\tilde{\varepsilon}^{2}
\right)
\text{ $P_{\theta}$-a.s.}.
\label{Gaussian_Predictivedistribution}
\end{eqnarray}
For the derivations of (\ref{Gaussian_Posterior}) and (\ref{Gaussian_Predictivedistribution}),
see Theorem 3.2 in \citet{Zhao(2000)}.
Let $v^{2}_{\varepsilon}$ and $v^{2}_{\varepsilon,\tilde{\varepsilon}}$ be defined by
\begin{eqnarray}
	v^{2}_{\varepsilon,\tilde{\varepsilon}}:=\frac{1}{1/\varepsilon^{2}+1/\tilde{\varepsilon}^{2}}
&\mathrm{and}&
	v^{2}_{\varepsilon}:=\varepsilon^{2},
	\label{def_v}
\end{eqnarray}
respectively.
Let $\tau^{*}(\varepsilon,\tilde{\varepsilon})=(\tau^{*}_{1}(\varepsilon,\tilde{\varepsilon}),\tau^{*}_{2}(\varepsilon,\tilde{\varepsilon}),\ldots)$ be
the infinite sequence of which the $i$-th coordinate for $i\in\mathbb{N}$ is defined by
\begin{align}
	\left(\tau^{*}_{i}(\varepsilon,\tilde{\varepsilon})\right)^{2}=\frac{1}{2}
\left[(v^{2}_{\varepsilon}-v^{2}_{\varepsilon,\tilde{\varepsilon}})
\sqrt{1+\frac{4}{2\lambda(\varepsilon,\tilde{\varepsilon}) a^{2}_{i}(v^{2}_{\varepsilon}-v^{2}_{\varepsilon,\tilde{\varepsilon}})}}
-(v^{2}_{\varepsilon}+v^{2}_{\varepsilon,\tilde{\varepsilon}})
\right]_{+},
\label{optimalvariance}
\end{align}
where $[t]_{+}=\mathop{\max}\{t,0\}$,
and
$\lambda(\varepsilon,\tilde{\varepsilon})$ is determined by
\begin{eqnarray*}
	\sum_{i=1}^{\infty}a_{i}^{2}\left(\tau^{*}_{i}(\varepsilon,\tilde{\varepsilon})\right)^{2}=B.
\end{eqnarray*}
Let $T(\varepsilon,\tilde{\varepsilon})$ be the number defined by
\begin{eqnarray}
	T(\varepsilon,\tilde{\varepsilon}):=\mathop{\sup}\{i:\text{$\tau^{*}_{i}(\varepsilon,\tilde{\varepsilon})$ is non-zero}\}
=\mathop{\sup}\left\{i:\frac{1}{\lambda(\varepsilon,\tilde{\varepsilon})a^{2}_{i}}>2\tilde{\varepsilon}^{2}\right\}.
\label{truncationnumber}
\end{eqnarray}

The following is the principal theorem of this section.
\begin{thm}
\label{AsymptoticallyMinimaxTheorem}
Let $d(\varepsilon)$ be $\lfloor 1/\varepsilon^{2}\rfloor$.
Assume that $0 < a_{1} \leq a_{2} \leq \ldots \nearrow \infty$.
If
$1/\tilde{\varepsilon}=\mathrm{O}(1/\varepsilon)$ as $\varepsilon\to 0$
and
$\log(1/\varepsilon^{2})\mathop{\sum}_{i=1}^{d(\varepsilon)}a_{i}^{4}\left(\tau^{*}_{i}(\varepsilon,\tilde{\varepsilon})\right)^{4}
=\mathrm{o}(1) \text{ as $\varepsilon\to 0$}$,
then
\begin{align*}
\mathop{\lim}_{\varepsilon\rightarrow 0}
\left[
\left\{\mathop{\inf}_{\widehat{Q}\in \mathcal{D}}\mathop{\sup}_{\theta\in\Theta(a,B)} R(\theta,\widehat{Q})\right\}
\bigg/
\mathop{\sum}_{i=1}^{T(\varepsilon,\tilde{\varepsilon})}\frac{1}{2}
\log\left(\frac{1+(\tau^{*}_{i}(\varepsilon,\tilde{\varepsilon}))^{2}/v^{2}_{\varepsilon,\tilde{\varepsilon}}}
{1+(\tau^{*}_{i}(\varepsilon,\tilde{\varepsilon}))^{2}/v^{2}_{\varepsilon}}\right)
\right]
=1.
\end{align*}
Further,
the Bayesian predictive distribution based on 
$\mathrm{G}_{\tau=\tau^{*}(\varepsilon,\tilde{\varepsilon})}$ is
asymptotically minimax:
\begin{eqnarray*}
	\mathop{\sup}_{\theta\in\Theta(a,B)}R(\theta,Q_{\mathrm{G}_{\tau=\tau^{*}(\varepsilon,\tilde{\varepsilon})}})=(1+\mathrm{o}(1))
\mathop{\inf}_{\widehat{Q}\in\mathcal{D}}
\mathop{\sup}_{\theta\in\Theta(a,B)}R(\theta,\widehat{Q})
\end{eqnarray*}
as $\varepsilon\to 0$.
\end{thm}
The proof is provided in the next subsection.

\subsection{Proof of the principal theorem of Section \ref{Sec:AsymptoticallyMinimaxTheorem}}
The proof of Theorem \ref{AsymptoticallyMinimaxTheorem}
requires five lemmas.
Because the parameter is infinite-dimensional,
we need Lemmas \ref{GaussianMeasures} and \ref{Lowerboundbysubproblem}
in addition to Theorem 4.2 in \citet{XuandLiang(2010)}.

The first lemma provides the explicit form of the Kullback--Leibler risk of the Bayesian predictive distribution $Q_{\mathrm{G}_{\tau}}$.
The proof is provided in Appendix \ref{Appendix:ProofofSection2}.

\begin{lem}
\label{GaussianMeasures}
If $\theta\in l_{2}$ and $\tau\in l_{2}$,
then 
$Q_{\mathrm{G}_{\tau}}(\cdot|X)$ and $Q_{\theta}$ are mutually absolutely continuous given $X=x$ $P_{\theta}$-a.s.~and
the Kullback--Leibler risk $R(\theta,Q_{\mathrm{G}_{\tau}})$ of the Bayesian predictive distribution $Q_{\mathrm{G}_{\tau}}$
is given by
\begin{align}
	R(\theta,Q_{\mathrm{G}_{\tau}})
=
\mathop{\sum}_{i=1}^{\infty}\left\{\frac{1}{2}\log\left(\frac{1+\tau^{2}_{i}/v^{2}_{\varepsilon,\tilde{\varepsilon}}}{1+\tau^{2}_{i}/v^{2}_{\varepsilon}}\right)
	+\frac{1}{2}\frac{v^{2}_{\varepsilon,\tilde{\varepsilon}}+\theta^{2}_{i}}{v^{2}_{\varepsilon,\tilde{\varepsilon}}+\tau^{2}_{i}}
	-\frac{1}{2}\frac{v^{2}_{\varepsilon}+\theta^{2}_{i}}{v^{2}_{\varepsilon}+\tau^{2}_{i}}
\right\}.
\label{KL_Bayes_productNormal}
\end{align}
\end{lem}

The second lemma provides the Bayesian predictive distribution that is minimax among the sub class of $\mathcal{D}$.
The proof is provided in Appendix \ref{Appendix:ProofofSection2}.

\begin{lem}
\label{LinearMinimax}
Assume that $0 < a_{1} \leq a_{2} \leq \cdots \nearrow \infty$.
Then,
for any $\varepsilon>0$ and any $\tilde{\varepsilon}>0$,
$T(\varepsilon,\tilde{\varepsilon})$ is finite
and
$\lambda(\varepsilon,\tilde{\varepsilon})$ is uniquely determined.
Further,
\begin{eqnarray*}
	\mathop{\inf}_{\tau\in l_{2}}\mathop{\sup}_{\theta\in\Theta(a,B)} R(\theta,Q_{\mathrm{G}_{\tau}})
	&=&\mathop{\sup}_{\theta\in\Theta(a,B)} \mathop{\inf}_{\tau\in l_{2}} R(\theta,Q_{\mathrm{G}_{\tau}})
\nonumber\\	
&=&\mathop{\sup}_{\theta\in\Theta(a,B)} R(\theta, Q_{\mathrm{G}_{\tau=\tau^{*}(\varepsilon,\tilde{\varepsilon})}})
\nonumber\\
&=&\mathop{\sum}_{i=1}^{T(\varepsilon,\tilde{\varepsilon})}\frac{1}{2}
\log\left(\frac{1+(\tau^{*}_{i}(\varepsilon,\tilde{\varepsilon}))^{2}/v^{2}_{\varepsilon,\tilde{\varepsilon}}}
{1+(\tau^{*}_{i}(\varepsilon,\tilde{\varepsilon}))^{2}/v^{2}_{\varepsilon}}\right).
\end{eqnarray*}
\end{lem}

The third lemma provides the upper bound of the minimax risk.

\begin{lem}
\label{upperboundMinimax}
Assume that $0 < a_{1} \leq a_{2} \leq \cdots \nearrow \infty$.
Then, 
for any $\varepsilon>0$ and any $\tilde{\varepsilon}>0$,
\begin{align*}
\mathop{\inf}_{\widehat{Q}\in \mathcal{D}}\mathop{\sup}_{\theta\in\Theta(a,B)} R(\theta,\widehat{Q})
\leq
\mathop{\sum}_{i=1}^{T(\varepsilon,\tilde{\varepsilon})}\frac{1}{2}
\log\left(\frac{1+(\tau^{*}_{i}(\varepsilon,\tilde{\varepsilon}))^{2}/v^{2}_{\varepsilon,\tilde{\varepsilon}}}
{1+(\tau^{*}_{i}(\varepsilon,\tilde{\varepsilon}))^{2}/v^{2}_{\varepsilon}}\right).
\end{align*}
\end{lem}
\begin{proof}
Since the class $\{Q_{\mathrm{G}_{\tau}}(\cdot|\cdot):\tau\in l_{2}\}$ is included in $\mathcal{D}$,
the result follows from Lemma \ref{LinearMinimax}.
\end{proof}

We introduce the notations for providing the lower bound of the minimax risk.
These notations are also used in Lemma \ref{Oracle}.
Fix an arbitrary positive integer $d$.
Let $\theta^{(d)}$ be $(\theta_{1},\ldots,\theta_{d})$.
Let $x^{(d)}$ be $(x_{1},\ldots,x_{d})$.
Let $P^{(d)}_{\theta^{(d)}}$ and $Q^{(d)}_{\theta^{(d)}}$ be $\mathop{\otimes}_{i=1}^{d}\mathcal{N}(\theta_{i},\varepsilon^{2})$
and $\mathop{\otimes}_{i=1}^{d}\mathcal{N}(\theta_{i},\tilde{\varepsilon}^{2})$, respectively.
Let $\Theta^{(d)}(a,B)$ be the $d$-dimensional parameter space defined by
\begin{eqnarray*}
\Theta^{(d)}(a,B):=\left\{\theta^{(d)}=(\theta_{1},\ldots,\theta_{d}):\mathop{\sum}_{i=1}^{d}a^{2}_{i}\theta^{2}_{i}\leq B\right\}.
\end{eqnarray*}
Let $R_{d}(\theta^{(d)},\widehat{Q}^{(d)}(\cdot;\cdot))$ be the $d$-dimensional Kullback--Leibler risk 
\begin{align*}
	R_{d}(\theta^{(d)},\widehat{Q}^{(d)}):=
	\int\int
	\log\frac{\mathrm{d}Q^{(d)}_{\theta^{(d)}}}{\mathrm{d}\widehat{Q}^{(d)}(\cdot;X^{(d)}=x^{(d)})}(y^{(d)})
	\mathrm{d}Q^{(d)}_{\theta^{(d)}}(y^{(d)})
	\mathrm{d}P^{(d)}_{\theta^{(d)}}(x^{(d)})
\end{align*}
of
predictive distribution $\widehat{Q}^{(d)}$ on $(\mathbb{R}^{d},\mathcal{R}^{d})$.
Let $R_{d}(\Theta^{(d)}(a,B))$ be the minimax risk
\begin{eqnarray*}
	R_{d}(\Theta^{(d)}(a,B)):=
	\mathop{\inf}_{\widehat{Q}^{(d)}\in\mathcal{D}^{(d)}}
\mathop{\sup}_{\theta^{(d)}\in\Theta^{(d)}(a,B)}
R_{d}(\theta^{(d)},\widehat{Q}^{(d)}),
\end{eqnarray*}
where $\mathcal{D}^{(d)}$ is $\{\mathbb{R}^{d}\to\mathcal{A}^{(d)}\}$ 
with the whole set
$\mathcal{A}^{(d)}$
of probability distributions on $(\mathbb{R}^{d},\mathcal{R}^{d})$.

The fourth lemma shows that the minimax risk in the infinite sequence model is bounded below by the minimax risk in the finite dimensional sequence model.
The proof is provided in Appendix \ref{Appendix:ProofofSection2}.
\begin{lem}
\label{Lowerboundbysubproblem}
Let $d$ be any positive integer.
Then, for any $\varepsilon>0$ and any $\tilde{\varepsilon}>0$,
\begin{eqnarray*}
\mathop{\inf}_{\widehat{Q}\in \mathcal{D}}\mathop{\sup}_{\theta\in\Theta(a,B)}R(\theta,\widehat{Q})
&\geq&
R_{d}(\Theta^{(d)}(a,B)).
\end{eqnarray*}
\end{lem}

The fifth lemma provides the asymptotic minimax risk in a high-dimensional sequence model.
It is due to \citet{XuandLiang(2010)}.

\begin{lem}[Theorem 4.2 in \citet{XuandLiang(2010)}]
\label{MinimaxLowerboundinsieve}
Let $\tau^{*}(\varepsilon,\tilde{\varepsilon})$ be defined by (\ref{optimalvariance}).
Let $T(\varepsilon,\tilde{\varepsilon})$ be defined by (\ref{truncationnumber}).
Let $d(\varepsilon)$ be $\lfloor 1/\varepsilon^{2}\rfloor$ where $\lfloor x\rfloor:=\max\{n\in\mathbb{Z}:n\leq x\}$.
Assume that $0 < a_{1}\leq a_{2} \leq \cdots \nearrow \infty$.
If
$1/\tilde{\varepsilon}=\mathrm{O}(1/\varepsilon)$ as $\varepsilon\to 0$
and
$\log(1/\varepsilon^{2})\mathop{\sum}_{i=1}^{d(\varepsilon)}a_{i}^{4}\left(\tau^{*}_{i}(\varepsilon,\tilde{\varepsilon})\right)^{4}
=\mathrm{o}(1) \text{ as $\varepsilon\to 0$}$,
then 
\begin{eqnarray*}
\mathop{\lim}_{\varepsilon\to 0}
\left[R_{d(\varepsilon)}(\Theta^{(d(\varepsilon))}(a,B))
	\bigg/
\mathop{\sum}_{i=1}^{T(\varepsilon,\tilde{\varepsilon})}\frac{1}{2}
\log\left(\frac{1+(\tau^{*}_{i}(\varepsilon,\tilde{\varepsilon}))^{2}/v^{2}_{\varepsilon,\tilde{\varepsilon}}}
{1+(\tau^{*}_{i}(\varepsilon,\tilde{\varepsilon}))^{2}/v^{2}_{\varepsilon}}\right)
\right]=1.
\end{eqnarray*}
\end{lem}

\vspace{4mm}

Based on these lemmas,
we present the proof of Theorem \ref{AsymptoticallyMinimaxTheorem}. 
\begin{proof}[Proof of Theorem \ref{AsymptoticallyMinimaxTheorem}]
From Lemma \ref{upperboundMinimax},
\begin{eqnarray*}
\inf_{\widehat{Q}\in\mathcal{D}}\sup_{\theta\in\Theta(a,B)}R(\theta,\widehat{Q})
\leq
\mathop{\sum}_{i=1}^{T(\varepsilon,\tilde{\varepsilon})}\frac{1}{2}
\log\left(\frac{1+(\tau^{*}_{i}(\varepsilon,\tilde{\varepsilon}))^{2}/v^{2}_{\varepsilon,\tilde{\varepsilon}}}
{1+(\tau^{*}_{i}(\varepsilon,\tilde{\varepsilon}))^{2}/v^{2}_{\varepsilon}}\right).
\end{eqnarray*}
From Lemma \ref{Lowerboundbysubproblem} with $d=\lfloor 1/\varepsilon^{2}\rfloor$ and Lemma \ref{MinimaxLowerboundinsieve},
\begin{eqnarray*}
\inf_{\widehat{Q}\in\mathcal{D}}\sup_{\theta\in\Theta(a,B)}R(\theta,\widehat{Q})
\geq
(1-\mathrm{o}(1))
\mathop{\sum}_{i=1}^{T(\varepsilon,\tilde{\varepsilon})}\frac{1}{2}
\log\left(\frac{1+(\tau^{*}_{i}(\varepsilon,\tilde{\varepsilon}))^{2}/v^{2}_{\varepsilon,\tilde{\varepsilon}}}
{1+(\tau^{*}_{i}(\varepsilon,\tilde{\varepsilon}))^{2}/v^{2}_{\varepsilon}}\right).
\end{eqnarray*}
This completes the proof.
\end{proof}

\subsection{Examples of asymptotically minimax predictive distributions}

In this subsection,
we provide the asymptotically minimax Kullback--Leibler risks 
and the asymptotically minimax predictive distributions
in the case that $\Theta(a,B)$ is a Sobolev ellipsoid
and
in the case that it is an exponential ellipsoid.

\subsubsection{The Sobolev ellipsoid}
The simplified Sobolev ellipsoid is
$\Theta_{\mathrm{Sobolev}}(\alpha,B)=\{\theta\in l_{2}:\sum_{i=1}^{\infty}i^{2\alpha}\theta^{2}_{i}\leq B\}$
with $\alpha>0$ and $B>0$.
We set $\tilde{\varepsilon}=\gamma\varepsilon$ for $\gamma>0$.
This setting is a slight generalization of Section 5 of \citet{XuandLiang(2010)},
in which
the asymptotic minimax Kullback--Leibler risk with $\gamma=1$ is obtained.

We expand $T:=T(\varepsilon,\tilde{\varepsilon})$ and $\tau^{*}(\varepsilon,\tilde{\varepsilon})$.
From the definition of $T$,
we have $2\lambda(\varepsilon,\tilde{\varepsilon})=\frac{1}{T^{2\alpha}\tilde{\varepsilon}^{2}}(1+\mathrm{o}(1))$.
Thus, we have
\begin{eqnarray*}
2B
&=&\frac{\varepsilon^{2}T^{2\alpha+1}}{\gamma^{2}+1}
\left[\int_{0}^{1}x^{2\alpha}\sqrt{1+4\gamma^{2}(\gamma^{2}+1)x^{-2\alpha}}\mathrm{d}x
-\frac{2\gamma^{2}+1}{2\alpha+1}\right]
(1+\mathrm{o}(1)),
\end{eqnarray*}
where we use the convergence of the Riemann sum $\sum_{i=1}^{T}r(i/T)1/T$ with 
the function
$r(x):=\sqrt{x^{4\alpha}+4\gamma^{2}(\gamma^{2}+1)x^{2\alpha}}$.
Then,
\begin{align}
	T(\varepsilon,\tilde{\varepsilon})&=
	\left(\frac{B}{\varepsilon^{2}}\right)^{1/(2\alpha+1)}
	\left[
	\frac{2(\gamma^{2}+1)}{\int_{0}^{1}x^{2\alpha}\sqrt{1+4\gamma^{2}(\gamma^{2}+1)x^{-2\alpha}}\mathrm{d}x
-\frac{2\gamma^{2}+1}{2\alpha+1}}\right]^{1/(2\alpha+1)}
\nonumber\\
&\quad\quad\times(1+\mathrm{o}(1))
\label{N_expansion}
\end{align}
and
\begin{eqnarray*}
	\left(\tau^{*}_{i}(\varepsilon,\tilde{\varepsilon})\right)^{2}=\frac{\varepsilon^{2}}{2}
	\left[
	\frac{1}{\gamma^{2}+1}\sqrt{1+4\gamma^{2}(\gamma^{2}+1)\left(\frac{i}{T}\right)^{-2\alpha}}
-\frac{2\gamma^{2}+1}{\gamma^{2}+1}
\right]_{+}
(1+\mathrm{o}(1)).
\end{eqnarray*}

\begin{figure}[h!]
	\begin{center}
		\includegraphics[width=80mm]{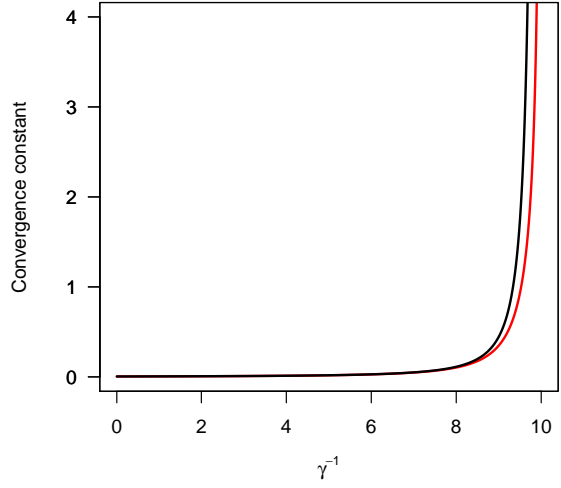}
	\end{center}
	\caption{Convergence constant 
		$\mathop{\lim}_{\varepsilon\rightarrow 0}\inf_{\widehat{Q}\in\mathcal{D}}\sup_{\theta\in\Theta_{\mathrm{Sobolev}}(\alpha,B)}2\varepsilon^{-2/(2\alpha+1)}R(\theta,\widehat{Q})$ 
with $\alpha=1$ and $B=1$:
The red line denotes the convergence constant where $\mathcal{A}$ is the whole set of probability distributions
and the black line denotes the convergence constant where $\mathcal{A}$ is the whole set of plug-in predictive distributions.
}
	\label{ConvergenceConstant_alpha1_gammavar}
\end{figure}

Thus, we obtain the asymptotically minimax risk
\begin{align}
	\mathop{\inf}_{\widehat{Q}\in\mathcal{D}}\mathop{\sup}_{\theta\in\Theta(\alpha,B)}&R(\theta,\widehat{Q})
\nonumber\\
&=\mathop{\sum}_{i=1}^{T}\frac{1}{2}
\log\left(\frac{1+(\tau^{*}_{i}(\varepsilon,\tilde{\varepsilon}))^{2}/v^{2}_{\varepsilon,\tilde{\varepsilon}}}
	{1+(\tau^{*}_{i}(\varepsilon,\tilde{\varepsilon}))^{2}/v^{2}_{\varepsilon}}
\right)
\nonumber\\
&=T\mathop{\sum}_{i=1}^{T}\frac{1}{2}
\log\left(
	1+\frac{1}{\gamma^{2}+\frac{2\gamma^{2}(\gamma^{2}+1)}{\sqrt{1+4\gamma^{2}(\gamma^{2}+1)(i/N)^{-2\alpha}}-(2\gamma^{2}+1)}}
\right)
\frac{1}{T}
\nonumber\\
&=\frac{T}{2}\int_{0}^{1}
\log\left(
	1+\frac{1}{\gamma^{2}+\frac{2\gamma^{2}(\gamma^{2}+1)}{\sqrt{1+4\gamma^{2}(\gamma^{2}+1)x^{-2\alpha}}-(2\gamma^{2}+1)}}
\right)
\mathrm{d}x(1+\mathrm{o}(1))
\nonumber\\
&=\left(\frac{B}{\varepsilon^{2}}\right)^{1/(2\alpha+1)} \mathcal{P}^{*}(1+\mathrm{o}(1)),
\end{align}
where
\begin{eqnarray*}
\mathcal{P}^{*}&=&
\frac{1}{2}
\left[\frac{2(\gamma^{2}+1)}{\int_{0}^{1}x^{2\alpha}\sqrt{1+4\gamma^{2}(\gamma^{2}+1)x^{-2\alpha}}\mathrm{d}x
-\frac{2\gamma^{2}+1}{2\alpha+1}}\right]^{1/(2\alpha+1)}
\nonumber\\
&&\times
\int_{0}^{1}
\log\left(
	1+\frac{1}{\gamma^{2}+\frac{2\gamma^{2}(\gamma^{2}+1)}{\sqrt{1+4\gamma^{2}(\gamma^{2}+1)x^{-2\alpha}}-(2\gamma^{2}+1)}}
\right)
\mathrm{d}x.
\end{eqnarray*}

We compare the Kullback--Leibler risk of the asymptotically minimax predictive distribution with
the Kullback--Leibler risk of the plug-in predictive distribution that is asymptotically minimax among all plug-in predictive distributions.
The latter is obtained using Pinsker's asymptotically minimax theorem for estimation (see \citet{Pinsker(1980)}).
We call the former and the latter risks the predictive and the estimative asymptotically minimax risks, respectively.
The orders of $\varepsilon^{-2}$ and $B$ in the predictive asymptotic minimax risk are both the $1/(2\alpha+1)$-th power.
These orders are the same as in the estimative asymptotically minimax risk.
However, the convergence constant $\mathcal{P}^{*}$ and 
the convergence constant in the estimative asymptotically minimax risk
are different.
Note that the convergence constant in the estimative asymptotically minimax risk 
is the Pinsker constant $(2\alpha+1)^{\frac{1}{2\alpha+1}}\left(\frac{\alpha}{\alpha+1}\right)^{2\alpha(2\alpha+1)}$ multiplied by $1/(2\gamma^{2})$.
Figure \ref{ConvergenceConstant_alpha1_gammavar}
shows
that the convergence constant $\mathcal{P}^{*}$ becomes smaller than
the convergence constant in the estimative asymptotically minimax risk
as $\gamma^{-1}$ increases.
\citet{XuandLiang(2010)} also pointed out this phenomenon when $\gamma=1$.

\subsubsection{The exponential ellipsoid}
The exponential ellipsoid is
$\Theta_{\mathrm{exp}}(\alpha,B)=\{\theta\in l_{2}:\mathop{\sum}_{i=1}^{\infty}\mathrm{e}^{2\alpha i} \theta^{2}_{i}\leq B\}$,
with $\alpha > 0$ and $B>0$.
We set $\tilde{\varepsilon}=\gamma\varepsilon$ for $\gamma>0$.

We expand $T:=T(\varepsilon,\tilde{\varepsilon})$ and $\tau^{*}(\varepsilon,\tilde{\varepsilon})$.
From the definition of $T$,
we have $2\lambda(\varepsilon,\tilde{\varepsilon})=\frac{1}{\mathrm{e}^{2\alpha T}\tilde{\varepsilon}^{2}}(1+\mathrm{o}(1))$.
Thus,
\begin{eqnarray*}
	2B
	&=&N\mathrm{e}^{2\alpha T}\frac{\varepsilon^{2}}{(\gamma^{2}+1)}r(\alpha,\gamma)(1+\mathrm{o}(1)),
\end{eqnarray*}
where $r(\alpha,\gamma)$ is a bounded term with respect to $N$.
Then,
\begin{eqnarray*}
	T(\varepsilon,\tilde{\varepsilon})&=&\frac{1}{\alpha}\log(\varepsilon^{-1})+\mathrm{o}(\log(\varepsilon^{-1}))
\end{eqnarray*}
and
\begin{eqnarray*}
	\left(\tau^{*}_{i}(\varepsilon,\tilde{\varepsilon})\right)^{2}
&=&
	\frac{\varepsilon^{2}}{2(\gamma^{2}+1)}
	\left[\sqrt{1+4\gamma^{2}(\gamma^{2}+1)\mathrm{e}^{-2\alpha(i-N)}}-(2\gamma^{2}+1)\right]_{+}(1+\mathrm{o}(1)).
\end{eqnarray*}
Thus,
we obtain the asymptotically minimax risk
\begin{eqnarray}
	&&\mathop{\inf}_{\widehat{Q}\in\mathcal{D}}\mathop{\sup}_{\theta\in\Theta_{\mathrm{exp}}(\alpha,B)}R(\theta,\widehat{Q})
\nonumber\\
&&=\frac{1}{2}\mathop{\sum}_{i=1}^{T}\log\left(1+\frac{1}{\gamma^{2}+\frac{2\gamma^{2}(\gamma^{2}+1)}{\sqrt{1+4\gamma^{2}(\gamma^{2}+1)\mathrm{e}^{-2\alpha(i-T)}}-(2\gamma^{2}+1)}}\right)(1+\mathrm{o}(1))
\nonumber\\
&&\sim \log(\varepsilon^{-1}).
\label{AsymptoticKL_exp}
\end{eqnarray}

We compare the predictive asymptotically minimax risk with the estimative asymptotically minimax risk
in the exponential ellipsoid.
From (\ref{AsymptoticKL_exp}),
\begin{eqnarray*}
	\mathop{\lim}_{\varepsilon\to 0}\left\{
		\mathop{\inf}_{\widehat{Q}\in\mathcal{D}}
\mathop{\sup}_{\theta\in\Theta_{\mathrm{exp}}(\alpha,B)}
R(\theta,\widehat{Q})
\bigg/\log(\varepsilon^{-1})
	\right\}
\leq\frac{\log(1+1/\gamma^{2})}{2\alpha}.
\end{eqnarray*}
From Pinsker's asymptotically minimax theorem,
\begin{eqnarray*}
	\mathop{\lim}_{\varepsilon\to 0}\left\{
	\mathop{\inf}_{\hat{\theta}}\mathop{\sup}_{\theta\in\Theta_{\mathrm{exp}}(\alpha,B)}
R\left(\theta,Q_{\hat{\theta}}\right)
\bigg/\log(\varepsilon^{-1})
	\right\}
=\frac{1}{2\gamma^{2}\alpha}.
\end{eqnarray*}
Thus, for any $\gamma>0$,
\begin{align*}
	\mathop{\lim}_{\varepsilon\to 0}&\left[
		\mathop{\inf}_{\widehat{Q}\in\mathcal{D}}\mathop{\sup}_{\theta\in\Theta_{\mathrm{exp}}(\alpha,B)}R(\theta,\widehat{Q})\bigg/\log(\varepsilon^{-1})
	\right]
\nonumber\\
&<
\mathop{\lim}_{\varepsilon\to 0}\left[
	\mathop{\inf}_{\hat{\theta}}\mathop{\sup}_{\theta\in\Theta_{\mathrm{exp}}(\alpha,B)}
R(\theta,Q_{\hat{\theta}})\bigg/\log(\varepsilon^{-1})
	\right].
\end{align*}
In an exponential ellipsoid,
the order of $\varepsilon$ in the predictive asymptotically minimax risk
is the same as that in the estimative asymptotically minimax risk.
The convergence constant in the predictive asymptotically minimax risk 
is strictly smaller than that in the estimative asymptotically minimax risk.

\begin{rem}
There are differences between the asymptotically minimax risks in the Sobolev and the exponential ellipsoids.
The constant $B$ has the same order in the asymptotically minimax risk 
as that of $\varepsilon^{-2}$
when the parameter space is the Sobolev ellipsoid.
In contrast,
the constant $B$ disappears
in the asymptotically minimax risk when the parameter space is the exponential ellipsoid.
\end{rem}

\section{Asymptotically minimax adaptive predictive distribution}
\label{Sec:Adaptivity}

In this section,
we show that the blockwise Stein predictive distribution
is asymptotically minimax adaptive
on the family of Sobolev ellipsoids.
Recall that
the Sobolev ellipsoid is
$\Theta_{\mathrm{Sobolev}}(\alpha,B)=\{\theta\in l_{2}: \sum_{i=1}^{\infty}i^{2\alpha}\theta_{i}^{2}\leq B\}$
with $\alpha>0$ and $B>0$.

\subsection{Principal theorem of Section \ref{Sec:Adaptivity}}

For the principal theorem,
we introduce a blockwise Stein predictive distribution and a weakly geometric blocks system.

A blockwise Stein predictive distribution for a set of blocks is constructed as follows.
Let $d$ be any positive integer.
We divide $\{1,\ldots,d\}$ into $J$ blocks:
$\{1,\cdots,d\}=\mathop{\cup}_{j=1}^{J}B_{j}$.
We denote the number of elements in each block $B_{j}$ by $b_{j}$.
Corresponding to the division into the blocks $\mathcal{B}(d):=\{B_{j}\}_{j=1}^{J}$,
we divide $\theta^{(d)}$ into $\theta_{B_{1}}=(\theta_{1},\ldots,\theta_{b_{1}})$,
$\cdots$,
and
$\theta_{B_{J}}=(\theta_{\sum_{j=1}^{J-1}{b_{j}}+1},\ldots,\theta_{d})$.
In the same manner,
we divide $X^{(d)}$ into $X_{B_{1}}$,$\cdots$, and $X_{B_{J}}$.
Let $h^{(d)}_{\mathcal{B}(d)}$ be the blockwise Stein prior
with the set of blocks $\mathcal{B}(d)$
defined by
\begin{eqnarray*}
	h^{(d)}_{\mathcal{B}(d)}(\theta^{(d)}):=\left(\mathop{\prod}_{j\in\{1,\ldots J\}:b_{j}>2}||\theta_{B_{j}}||^{2-b_{j}}\right),
\end{eqnarray*}
where $||\cdot||$ is the square norm.
We define the blockwise Stein predictive distribution with the set of blocks $\mathcal{B}(d)$ as
\begin{align}
	Q_{h^{(d)}_{\mathcal{B}(d)}}(\cdot|X):=&
	\left(
	\int \otimes_{i=1}^{d}\mathcal{N}(\theta_{i},\tilde{\varepsilon}^{2}) 
	h^{(d)}_{\mathcal{B}(d)}(\theta^{(d)}|X^{(d)})\mathrm{d}\theta^{(d)}
	\right)
	\otimes
	\left(
	\mathop{\otimes}_{i=d+1}^{\infty}\mathcal{N}(0,\tilde{\varepsilon}^{2})
	\right),\quad\quad
	\label{def_blockwiseStein}
\end{align}
where $h^{(d)}_{\mathcal{B}(d)}(\theta^{(d)}|X^{(d)})$ 
is the posterior density of $h^{(d)}_{\mathcal{B}(d)}(\theta^{(d)})$.
In regard to estimation,
\citet{BrownandZhao(2009)} discussed the behavior of the Bayes estimator based on the blockwise Stein prior.

The weakly geometric blocks (WGB) system is introduced as follows.
The WGB system
$\mathcal{B}^{*}_{\varepsilon}:=\{B^{*}_{\varepsilon,j}\}_{j=1}^{J(\varepsilon)}$
with cardinalities $\{b^{*}_{\varepsilon,j}\}_{j=1}^{J(\varepsilon)}$
is the division of $\{1,\ldots,d(\varepsilon)\}$, where $d(\varepsilon)=\lfloor 1/\varepsilon^{2}\rfloor$.
It
is defined by
\begin{eqnarray}
	b^{*}_{\varepsilon,1}&=&\lceil \rho_{\varepsilon}^{-1}\rceil,
\nonumber\\
b^{*}_{\varepsilon,2}&=&\lfloor b^{*}_{\varepsilon,1}(1+\rho_{\varepsilon})\rfloor,
\nonumber\\
\cdots
\nonumber\\
b^{*}_{\varepsilon,J(\varepsilon)-1}&=&\lfloor b^{*}_{\varepsilon,1}(1+\rho_{\varepsilon})^{J(\varepsilon)-2}\rfloor,
\nonumber\\
b^{*}_{\varepsilon,J(\varepsilon)}&=&d(\varepsilon)-\mathop{\sum}_{j=1}^{J(\varepsilon)-1}b^{*}_{\varepsilon,j},
\label{def_WGB}
\end{eqnarray}
where 
$\rho_{\varepsilon}=(\log(1/\varepsilon))^{-1}$
and
$J(\varepsilon)=\mathop{\min}\{m: b^{*}_{\varepsilon,1}+\mathop{\sum}_{j=2}^{m}
\lfloor b^{*}_{\varepsilon,1}(1+\rho_{\varepsilon})^{j-1}\rfloor\geq d(\varepsilon)\}$.
The WGB system has been used for the construction of an asymptotically minimax adaptive estimator;
see \citet{CavalierandTsybakov(2001)} and \citet{Tsybakov(2009)}.

\vspace{4mm}

The following is the principal theorem.
Let $d(\varepsilon)$ be $\lfloor 1/\varepsilon^{2} \rfloor$.
Let $\mathcal{B}^{*}_{\varepsilon}$ be the WGB system defined by (\ref{def_WGB})
with cardinalities $\{b^{*}_{\varepsilon,j}\}_{j=1}^{J(\varepsilon)}$.
Let
$Q_{h^{\left(d(\varepsilon)\right)}_{\mathcal{B}^{*}_{\varepsilon}}}$
be
the blockwise Stein predictive distribution 
with the WGB system $\mathcal{B}^{*}_{\varepsilon}$
defined by (\ref{def_blockwiseStein}).

\begin{thm}
\label{SobolevAdaptive}
If $\tilde{\varepsilon}=\gamma\varepsilon$ for some $\gamma>0$,
then
for any $\alpha>0$ and for any $B>0$,
\begin{eqnarray*}
	\lim_{\varepsilon\rightarrow 0}\left[\mathop{\sup}_{\theta\in\Theta_{\mathrm{Sobolev}}(\alpha,B)}
		R\left(\theta,Q_{h^{\left(d(\varepsilon)\right)}_{\mathcal{B}^{*}_{\varepsilon}}}\right)
	\bigg/\mathop{\inf}_{\widehat{Q}\in\mathcal{D}}\mathop{\sup}_{\theta\in\Theta_{\mathrm{Sobolev}}(\alpha,B)}R(\theta,\widehat{Q})\right]=1.
\end{eqnarray*}
\end{thm}

The proof is provided in Subsection \ref{Subsec:Proof}.
An inequality related to the Bayesian predictive distribution based on Stein's prior 
that we will use in the proof of Theorem \ref{SobolevAdaptive}
will be shown in Subsection \ref{Subsec:Oracle}.
In Subsection \ref{Subsec:Proof},
we introduce several lemmas and provide the proof of Theorem \ref{SobolevAdaptive}.

\subsection{Oracle inequality of the Bayesian predictive distribution based on Stein's prior}\label{Subsec:Oracle}
Before considering the proof of Theorem \ref{SobolevAdaptive},
we show an oracle inequality related to Stein's prior for $d>2$
that is useful outside of the proof of Theorem \ref{SobolevAdaptive}.
Recall that the $d$-dimensional Kullback--Leibler risk
$R_{d}(\theta^{(d)},\widehat{Q}^{(d)})$
of the predictive distribution $\widehat{Q}^{(d)}$ on $(\mathbb{R}^{d},\mathcal{R}^{d})$
is defined by
\begin{align*}
	R_{d}(\theta^{(d)},\widehat{Q}^{(d)})=
\int \int
\log\frac{\mathrm{d}Q^{(d)}_{\theta^{(d)}}}{\mathrm{d}\widehat{Q}^{(d)}(\cdot;X^{(d)}=x^{(d)})}(y^{(d)})
\mathrm{d}Q^{(d)}_{\theta^{(d)}}(x^{(d)})\mathrm{d}P^{(d)}_{\theta^{(d)}}(y^{(d)})
\end{align*}
where
$P^{(d)}_{\theta^{(d)}}$ and $Q^{(d)}_{\theta^{(d)}}$ are 
$\otimes_{i=1}^{d}\mathcal{N}(\theta_{i},\varepsilon^{2})$ and $\otimes_{i=1}^{d}\mathcal{N}(\theta_{i},\tilde{\varepsilon}^{2})$,
respectively.
For a positive integer $d>2$,
let $Q^{(d)}_{h^{(d)}}$ be the Bayesian predictive distribution on $\mathbb{R}^{d}$
based on the $d$-dimensional Stein's prior $h^{(d)}(\theta^{(d)})=||\theta^{(d)}||^{2-d}$.
\begin{lem}
\label{Oracle}
Let $d$ be any positive integer such that $d>2$.
For any $\theta^{(d)}\in \mathbb{R}^{d}$,
\begin{eqnarray}
	R_{d}(\theta^{(d)},Q^{(d)}_{h^{(d)}})
&\leq&
	\log\left(\frac{v^{2}_{\varepsilon}}{v^{2}_{\varepsilon,\tilde{\varepsilon}}}\right)
+\frac{d}{2}\log\left(
	\frac{1+ (||\theta^{(d)}||^{2}/d)/v^{2}_{\varepsilon,\tilde{\varepsilon}}}
	{1+(||\theta^{(d)}||^{2}/d)/v^{2}_{\varepsilon}}
\right).\quad\quad
\label{oracleineq}
\end{eqnarray}
\end{lem}

\begin{rem}
We call inequality (\ref{oracleineq}) an oracle inequality of Stein's prior for the following reason.
By the same calculation in (\ref{Min_tau}) in the proof of Lemma \ref{LinearMinimax},
the second term on the right hand side of inequality (\ref{oracleineq})
is 
the oracle Kullback--Leibler risk,
that is,
the minimum of the Kullback--Leibler risk
in the case that the action space is
\begin{align*}
\left\{\otimes_{i=1}^{d}\mathcal{N}
	\left(\frac{(1/\varepsilon^{2})}{(1/\varepsilon^{2}+1/\tau^{2})}X_{i},1/(1/\varepsilon^{2}+1/\tau^{2})+\tilde{\varepsilon}^{2}\right):\tau^{2}\in[0,\infty)
	\right\},
\end{align*}
and 
in the case that we are permitted to use the value of the true parameter $\theta^{(d)}$.
Therefore,
Lemma \ref{Oracle} tells us that
the Kullback--Leibler risk of
the $d$-dimensional Bayesian predictive distribution based on Stein's prior
is bounded above by a constant independent of $d$
plus the oracle Kullback--Leibler risk.
\end{rem}

\begin{proof}[Proof of Lemma \ref{Oracle}]
First,
\begin{align*}
	R_{d}&(\theta^{(d)},Q^{(d)}_{u^{(d)}})-R_{d}(\theta^{(d)},Q^{(d)}_{h^{(d)}})
\nonumber\\
&=\frac{1}{2}\int^{v^{2}_{\varepsilon}}_{v^{2}_{\varepsilon,\tilde{\varepsilon}}}
\frac{1}{v^{2}}
\left\{\mathrm{E}_{v}||\theta^{(d)}-\hat{\theta}^{(d)}_{u^{(d)}}(X^{(d)})||^{2}
-\mathrm{E}_{v}||\theta^{(d)}-\hat{\theta}^{(d)}_{h^{(d)}}(X^{(d)})||^{2}
\right\}\mathrm{d}v,
\end{align*}
where
$Q^{(d)}_{u^{(d)}}$
is the Bayesian predictive distribution based on the uniform prior $u^{(d)}(\theta^{(d)}):=1$,
and
$\hat{\theta}^{(d)}_{u^{(d)}}$ and $\hat{\theta}^{(d)}_{h^{(d)}}$ 
are the Bayes estimators based on the uniform prior $u^{(d)}$ 
and based on Stein's prior $h^{(d)}$,
respectively.
Here
$\mathrm{E}_{v}$ is the expectation of $X^{(d)}$ with respect to the $d$-dimensional Gaussian distribution
with mean $\theta^{(d)}$ and covariance matrix $vI_{d}$.
For the proof of the identity,
see \citet{BrownGeorgeXu(2008)}.

Second,
\begin{eqnarray*}
	\mathrm{E}_{v}||\theta^{(d)}-\hat{\theta}^{(d)}_{h^{(d)}}(X^{(d)})||^{2}
\leq
\mathrm{E}_{v}||\theta^{(d)}-\hat{\theta}^{(d)}_{\mathrm{JS}}(X^{(d)})||^{2}
\leq
2v+\frac{dv||\theta^{(d)}||^2}{dv+||\theta^{(d)}||^{2}},
\end{eqnarray*}
where $\hat{\theta}^{(d)}_{\mathrm{JS}}$ is the James--Stein estimator.
For the first inequality,
see \citet{Kubokawa(1991)}.
For the second inequality,
see e.g.~Theorem 7.42 in \citet{Wasserman(2007)}.
Thus,
we have
\begin{align*}
	R_{d}(\theta^{(d)},&Q^{d)}_{u^{(d)}})-R_{d}(\theta^{(d)},Q^{(d)}_{h^{(d)}})
\nonumber\\
&\geq
\frac{1}{2}\int^{v^{2}_{\varepsilon}}_{v^{2}_{\varepsilon,\tilde{\varepsilon}}}\frac{d}{v}\mathrm{d}v
-\frac{1}{2}\int^{v^{2}_{\varepsilon}}_{v^{2}_{\varepsilon,\tilde{\varepsilon}}}
\frac{1}{v^{2}}
\left\{2v+\frac{dv||\theta^{(d)}||^{2}}{dv+||\theta^{(d)}||^{2}}\right\}\mathrm{d}v
\nonumber\\
&=
\frac{d}{2}\log\left(\frac{v^{2}_{\varepsilon}}{v^{2}_{\varepsilon,\tilde{\varepsilon}}}\right)
-\frac{1}{2}\int^{v^{2}_{\varepsilon}}_{v^{2}_{\varepsilon,\tilde{\varepsilon}}}
\frac{1}{v^{2}}
\left\{2v+\frac{dv||\theta^{(d)}||^{2}}{dv+||\theta^{(d)}||^{2}}\right\}\mathrm{d}v.
\end{align*}

Since
\begin{align*}
	R_{d}(\theta^{(d)},Q^{(d)}_{u^{(d)}})=\frac{d}{2}\log\left(\frac{v^{2}_{\varepsilon}}{v^{2}_{\varepsilon,\tilde{\varepsilon}}}\right),
\end{align*}
it follows that
\begin{align*}
	R_{d}(\theta^{(d)},Q^{(d)}_{h^{(d)}})
&\leq\frac{1}{2}\int^{v^{2}_{\varepsilon}}_{v^{2}_{\varepsilon,\tilde{\varepsilon}}}
\frac{1}{v^{2}}
\left\{2v+\frac{dv||\theta^{(d)}||^{2}}{dv+||\theta^{(d)}||^{2}}\right\}\mathrm{d}v
\nonumber\\
&=\log\left(\frac{v^{2}_{\varepsilon}}{v^{2}_{\varepsilon,\tilde{\varepsilon}}}\right)
+\frac{d}{2}\log\left(
	\frac{1+ (||\theta^{(d)}||^{2}/d)/v^{2}_{\varepsilon,\tilde{\varepsilon}}}
	{1+(||\theta^{(d)}||^{2}/d)/v^{2}_{\varepsilon}}
\right).
\end{align*}
Here, we use
\begin{eqnarray*}
	\frac{d}{v^{2}}\frac{v||\theta^{(d)}||^{2}}{dv+||\theta^{(d)}||^{2}}
	&=&
	d\left(\frac{1}{v}-\frac{1}{v+||\theta^{(d)}||^{2}/d}\right).
\end{eqnarray*}
\end{proof}

\vspace{4mm}

\begin{rem}
As a corollary of Lemma \ref{Oracle},
we show that the Bayesian predictive distribution based on Stein's prior is
asymptotically minimax adaptive
in the family of $\mathcal{L}_{2}$-balls
$\{\Theta_{\mathcal{L}_{2}}(d,B):B>0\}$.
The $\mathcal{L}_{2}$-ball is
$\Theta_{\mathcal{L}_{2}}(d,B)=\{\theta\in l_{2}:\sum_{i=1}^{d}\theta^{2}_{i}\leq B, \theta_{d+1}=\theta_{d+2}=\ldots=0 \}$.
Note that another type of an asymptotically minimax adaptive predictive distribution in the family of $\mathcal{L}_{2}$-balls has been investigated by \citet{XuandZhou(2011)}.
\end{rem}

\begin{lem}
\label{L2adaptive}
Let $d(\varepsilon)$ be $\lfloor1/\varepsilon^{2}\rfloor$.
If $\tilde{\varepsilon}=\gamma \varepsilon$ for some $\gamma>0$,
then
for any $B>0$,
the blockwise Stein predictive distribution
$Q_{h^{(d(\varepsilon))}_{\mathcal{S}_{\varepsilon}}}$
with the single block $\mathcal{S}_{\varepsilon}=\{\{1,\ldots,d(\varepsilon)\}\}$
on $\mathbb{R}^{\infty}$
satisfies
\begin{eqnarray*}
\lim_{\varepsilon\to\infty}
&\left[
	\mathop{\sup}_{\theta\in\Theta_{\mathcal{L}_{2}}(d(\varepsilon),B)}
	R\left(\theta,Q_{h^{(d(\varepsilon))}_{\mathcal{S}_{\varepsilon}}}\right)
\bigg{/}
\mathop{\inf}_{\widehat{Q}\in \mathcal{D}}
\mathop{\sup}_{\theta\in\Theta_{\mathcal{L}_{2}}(d(\varepsilon),B)}
R(\theta,\widehat{Q})
\right]=1.
\end{eqnarray*}
\end{lem}

\begin{proof}
From Lemma \ref{Oracle},
\begin{align*}
&\lim_{\varepsilon\to 0}
	\left[
		\frac{
	\mathop{\sup}_{\theta\in\Theta_{\mathcal{L}_{2}}(d(\varepsilon),B)}
	R\left(\theta,Q_{h^{(d(\varepsilon))}_{\mathcal{S}_{\varepsilon}}}
\right)}{d(\varepsilon)}
\right]
\nonumber\\
&\leq
\lim_{\varepsilon\to 0} \left[\frac{1}{d(\varepsilon)}
\log \left(\frac{v^{2}_{\varepsilon}}{v^{2}_{\varepsilon,\tilde{\varepsilon}}}\right)\right]
+\frac{1}{2}
\lim_{\varepsilon\to 0}
\left[
\mathop{\sup}_{\theta\in \Theta_{\mathcal{L}_{2}}(d(\varepsilon),B)}
\log\left(
\frac{1+ (||\theta^{(d)}||^{2}/d)/v^{2}_{\varepsilon,\tilde{\varepsilon}}}
	{1+(||\theta^{(d)}||^{2}/d)/v^{2}_{\varepsilon}}
\right)
\right].
\end{align*}
Note that
\begin{eqnarray*}
	\mathop{\sup}_{\theta\in \Theta_{\mathcal{L}_{2}}(d(\varepsilon),B)}
	\log\left(
	\frac{1+ (||\theta^{(d)}||^{2}/d)/v^{2}_{\varepsilon,\tilde{\varepsilon}}}
	{1+(||\theta^{(d)}||^{2}/d)/v^{2}_{\varepsilon}}
\right)
&=&
	\log \left(
	\frac{1+ (B/d)/v^{2}_{\varepsilon,\tilde{\varepsilon}}}
	{1+(B/d)/v^{2}_{\varepsilon}}
\right).
\end{eqnarray*}
Thus, we have
\begin{eqnarray*}
\lim_{\varepsilon\to 0}
\left[
	\frac{1}{d(\varepsilon)}
\mathop{\sup}_{\theta\in\Theta_{\mathcal{L}_{2}}(d(\varepsilon),B)}
R\left(\theta,Q_{h^{(d(\varepsilon))}_{\mathcal{S}_{\varepsilon}}}
\right)
\right]
&=&\frac{1}{2}\log\left(1+\frac{B}{\gamma^{2}(B+1)}\right).
\end{eqnarray*}

Since from Theorem 4.2 in \citet{XuandLiang(2010)} we have
\begin{eqnarray*}
\lim_{\varepsilon\to 0}
\left[
	\frac{1}{d(\varepsilon)}
\mathop{\inf}_{\widehat{Q}\in \mathcal{D}}
\mathop{\sup}_{\theta\in\Theta_{\mathcal{L}_{2}}(d(\varepsilon),B)}
R(\theta,\widehat{Q})
\right]
&=&\frac{1}{2}\log\left(1+\frac{B}{\gamma^{2}(B+1)}\right),
\end{eqnarray*}
the proof is complete.
\end{proof}

\subsection{Proof of the principal theorem of Section \ref{Sec:Adaptivity}}
\label{Subsec:Proof}

In this subsection,
we provide the proof of Theorem \ref{SobolevAdaptive}.
The proof consists of the following two steps.
First,
in Lemma \ref{Oracle_blockwiseStein},
we examine the properties of the blockwise Stein predictive distribution with a set of blocks.
The proof of Lemma \ref{Oracle_blockwiseStein}
requires Lemma \ref{MimicMonotoneclass}.
Second,
we show that
the blockwise Stein predictive distribution with the weakly geometric blocks system 
is asymptotically minimax adaptive on the family of Sobolev ellipsoids,
using Lemma \ref{Oracle_blockwiseStein} and the property of the WGB system (Lemma \ref{WGBproperty}).

For the proof,
we introduce two subspaces of the decision space.
For a given set of blocks $\mathcal{B}(d)=\{B_{j}\}_{j=1}^{J}$,
let $\mathcal{G}_{\mathrm{BW}}(\mathcal{B}(d))$ be
\begin{eqnarray*}
	\{X\to Q_{\mathrm{G}_{\tau}}(\cdot|X):\text{$\tau$ is equal to $(\tau^{(d)},0,0,\ldots)$ with } \tau^{(d)}\in \mathcal{T}_{\mathrm{BW}}\},
\end{eqnarray*}
where $\mathcal{T}_{\mathrm{BW}}=\{\tau^{(d)}:\text{ for $b\in\{1,\ldots,J\}$, } \tau_{i} \text{ is constant for } i\in B_{b}\}$
and
let $\mathcal{G}_{\mathrm{mon}}(\mathcal{B}(d))$ be
\begin{eqnarray*}
	\{X\to Q_{\mathrm{G}_{\tau}}(\cdot|X):\text{$\tau$ is equal to $(\tau^{(d)},0,0,\ldots)$ with } \tau^{(d)}\in \mathcal{T}_{\mathrm{mon}}\},
\end{eqnarray*}
where $\mathcal{T}_{\mathrm{mon}}=\{\tau^{(d)}:\tau_{1} \geq \tau_{2} \geq \cdots \geq \tau_{d} \geq 0 \}$.

Although
the decision space $\mathcal{G}_{\mathrm{BW}}(\mathcal{B}(d))$ is included in the decision space $\mathcal{G}_{\mathrm{mon}}(\mathcal{B}(d))$,
the following lemma states that if the growth rate of the numbers in each block in $\mathcal{B}(d)$ is controlled,
then the infimum of the Kullback--Leibler risk among $\mathcal{G}_{\mathrm{mon}}(\mathcal{B}(d))$
is bounded by 
a constant plus
a constant multiple of
the infimum of the Kullback--Leibler risk among $\mathcal{G}_{\mathrm{BW}}(\mathcal{B}(d))$.
The proof is provided in Appendix \ref{Appendix:ProofofSection4}.

\begin{lem}
\label{MimicMonotoneclass}
Let $d$ be any positive integer.
Let $\mathcal{B}(d)=\{B_{j}\}_{j=1}^{J}$ be a set of blocks whose cardinalities satisfy
\begin{eqnarray*}
\mathop{\mathrm{max}}_{1\leq j \leq J-1} \frac{b_{j+1}}{b_{j}} \leq 1+\eta
\end{eqnarray*}
for some $\eta>0$.
Then,
for any $\theta\in l_{2}$,
\begin{align*}
	\mathop{\inf}_{\widehat{Q}\in\mathcal{G}_{\mathrm{BW}}(\mathcal{B}(d))}
	R(\theta,\widehat{Q})
&\leq
(1+\eta)\mathop{\inf}_{\widehat{Q}\in\mathcal{G}_{\mathrm{mon}}(\mathcal{B}(d))}
R(\theta,\widehat{Q})
+\frac{b_{1}}{2}\log\left(\frac{v^{2}_{\varepsilon}}{v^{2}_{\varepsilon,\tilde{\varepsilon}}}\right).
\end{align*}
\end{lem}

The following lemma states the relationship between 
the Kullback--Leibler risk of the blockwise Stein predictive distribution 
and that of the predictive distribution in $\mathcal{G}_{\mathrm{mon}}(\mathcal{B}(d))$.
The proof is provided in Appendix \ref{Appendix:ProofofSection4}.

\begin{lem}
\label{Oracle_blockwiseStein}
Let $d$ be any positive integer.
Let $\mathcal{B}(d)=\{B_{j}\}_{j=1}^{J}$ be a set of blocks whose cardinalities satisfy
\begin{eqnarray*}
\mathop{\max}_{1\leq j\leq J-1}\frac{b_{j+1}}{b_{j}} \leq 1+\eta
\end{eqnarray*}
for some $\eta>0$.
Let $Q_{h^{(d)}_{\mathcal{B}(d)}}$ be the blockwise Stein predictive distribution 
with the set of blocks $\mathcal{B}(d)$
defined by (\ref{def_blockwiseStein}).
Then, 
for any $\theta\in l_{2}$,
\begin{eqnarray}
	R\left(\theta,Q_{h^{(d)}_{\mathcal{B}(d)}}\right)
	\leq
	(1+\eta)\inf_{ \widehat{Q}\in \mathcal{G}_{\mathrm{mon}}(\mathcal{B}(d))}
	R(\theta,\widehat{Q})
	+\left(J+\frac{b_{1}}{2}\right)\log\left(\frac{v^{2}_{\varepsilon}}{v^{2}_{\varepsilon,\tilde{\varepsilon}}}\right).\quad\quad
\label{oracleineq_blockwiseStein}
\end{eqnarray}
\end{lem}

The following lemma states that the WBG system satisfies the assumption in Lemmas \ref{MimicMonotoneclass} and \ref{Oracle_blockwiseStein}.
The proof is due to \citet{Tsybakov(2009)}.

\begin{lem}[e.g.,~Lemma 3.12 in \citet{Tsybakov(2009)}]
\label{WGBproperty}
Let $d(\varepsilon)$ be $\lfloor 1/\varepsilon^{2}\rfloor$.
Let $\mathcal{B}^{*}_{\varepsilon}=\{B^{*}_{\varepsilon,j}\}_{j=1}^{J(\varepsilon)}$
be the WGB system defined by (\ref{def_WGB}) with cardinalities
$\{b^{*}_{\varepsilon,j}\}_{j=1}^{J(\varepsilon)}$.
Then, there exist $0<\varepsilon_{0}<1$ and $C_{0}>0$ such that
\begin{eqnarray*}
	J(\varepsilon) \leq C_{0} \log^2(1/v_{\varepsilon}) \text{ for any } \varepsilon\in(0,\varepsilon_{0})
\end{eqnarray*}
and
\begin{eqnarray*}
	\mathop{\max}_{1\leq i \leq J(\varepsilon)-1}\frac{b^{*}_{\varepsilon,i+1}}{b^{*}_{\varepsilon,i}}
\leq 1+3\rho_{\varepsilon} \text{ for any } \varepsilon\in(0,\varepsilon_{0}).
\end{eqnarray*}
\end{lem}

\vspace{4mm}
Based on these lemmas, we provide the proof of Theorem \ref{SobolevAdaptive}.

\begin{proof}[Proof of Theorem \ref{SobolevAdaptive}]
First,
since the WGB system $\mathcal{B}^{*}_{\varepsilon}$ satisfies the assumption in Lemma \ref{Oracle_blockwiseStein},
it follows
from Lemma \ref{Oracle_blockwiseStein}
that
for $0<\varepsilon<\varepsilon_{0}$,
\begin{align*}
	\mathop{\sup}_{\theta\in\Theta_{\mathrm{Sobolev}}(\alpha,B)}
R\left(\theta,Q_{h^{\left(d(\varepsilon)\right)}_{\mathcal{B}^{*}_{\varepsilon}}}\right)
\leq&
(1+3\rho_{\varepsilon})\mathop{\sup}_{\theta\in\Theta_{\mathrm{Sobolev}}(\alpha,B)}
\inf_{\widehat{Q}\in \mathcal{G}_{\mathrm{mon}}(\mathcal{B}^{*}_{\varepsilon})}R(\theta,\widehat{Q})
\nonumber\\
&+\log\left(\frac{v^{2}_{\varepsilon}}{v^{2}_{\varepsilon,\tilde{\varepsilon}}}\right)
\left\{C_{0}\log^{2}(1/v_{\varepsilon})+\frac{\rho_{\varepsilon}^{-1}+1}{2}\right\}.
\end{align*}

Second,
we show that
the asymptotically minimax predictive distribution 
$Q_{\mathrm{G}_{\tau=\tau^{*}(\varepsilon,\tilde{\varepsilon})}}$
in Theorem \ref{AsymptoticallyMinimaxTheorem}
is also characterized as follows: for a sufficiently small $\varepsilon>0$,
\begin{eqnarray*}
	\mathop{\sup}_{\theta\in\Theta_{\mathrm{Sobolev}}(\alpha,B)}
	\mathop{\inf}_{\widehat{Q}\in \mathcal{G}_{\mathrm{mon}}(\mathcal{B}^{*}_{\varepsilon})}
	R(\theta,\widehat{Q})
=\mathop{\sup}_{\theta\in\Theta_{\mathrm{Sobolev}}(\alpha,B)}
R(\theta,Q_{G_{\tau=\tau^{*}(\varepsilon,\tilde{\varepsilon})}}).
\end{eqnarray*}
It suffices to show that
the Bayesian predictive distribution $Q_{\mathrm{G}_{\tau=\tau^{*}(\varepsilon,\tilde{\varepsilon})}}$
is included in $\mathcal{G}_{\mathrm{mon}}(\mathcal{B}^{*}_{\varepsilon})$
for a sufficiently small $\varepsilon>0$.
This is proved as follows.
Recall that $T(\varepsilon,\tilde{\varepsilon})$ defined by (\ref{truncationnumber})
is the maximal index of which $\tau^{*}_{i}(\varepsilon,\tilde{\varepsilon})$ defined by (\ref{optimalvariance})
is non-zero.
From the expansion of $T(\varepsilon,\tilde{\varepsilon})$
given in (\ref{N_expansion}),
for a sufficiently small $\varepsilon>0$,
for $i>A\varepsilon^{-2/(2\alpha+1)}$ with some constant $A$,
$\left(\tau^{*}_{i}(\varepsilon,\tilde{\varepsilon})\right)^{2}$
vanishes.
Since $\varepsilon^{-2/(2\alpha+1)} < \varepsilon^{-2}$ for $\varepsilon<1$,
the Bayesian predictive distribution $Q_{\mathrm{G}_{\tau=\tau^{*}(\varepsilon,\tilde{\varepsilon})}}$
is included in $\mathcal{G}_{\mathrm{mon}}(\mathcal{B}^{*}_{\varepsilon})$
for a sufficiently small $\varepsilon>0$.

Combining the first argument with the second argument yields
\begin{align*}
	\mathop{\sup}_{\theta\in\Theta_{\mathrm{Sobolev}}(\alpha,B)}
	&R\left(\theta,Q_{h^{\left(d(\varepsilon)\right)}_{\mathcal{B}^{*}_{\varepsilon}}}\right)
\nonumber\\
	&\leq
(1+3\rho_{\varepsilon})\mathop{\sup}_{\theta\in\Theta_{\mathrm{Sobolev}}(\alpha,B)} R(\theta,Q_{\mathrm{G}_{\tau=\tau^{*}(\varepsilon,\tilde{\varepsilon})}})
	\nonumber\\
	&\quad+C_{0}\log\left(\frac{v^{2}_{\varepsilon}}{v^{2}_{\varepsilon,\tilde{\varepsilon}}}\right)\log^{2}(1/v_{\varepsilon})+\mathrm{o}(1)
\nonumber\\
&=(1+3\rho_{\varepsilon})
\left(\frac{B}{\varepsilon^{2}}\right)^{1/(2\alpha+1)}
\mathcal{P}^{*}
\nonumber\\
&\quad
\times
\left(1+\frac{C_{0}\log(v^{2}_{\varepsilon}/v^{2}_{\varepsilon,\tilde{\varepsilon}})}{\mathcal{P}^{*}}
\varepsilon^{2/(2\alpha+1)}B^{-1/(2\alpha+1)}\log^{2}(1/v_{\varepsilon})\right)+\mathrm{o}(1)
\nonumber\\
&=\mathcal{P}^{*}
\left(\frac{B}{\varepsilon^{2}}\right)^{1/(2\alpha+1)}
(1+\mathrm{o}(1)).
\end{align*}
This completes the proof.
\end{proof}

\section{Numerical experiments}\label{Sec:Numerical}

In Subsection \ref{Subsec:Exactsampling},
we provide an exact sampling method for the blockwise Stein predictive distribution.
In Subsection \ref{Subsec:Numerical},
we provide two numerical experiments concerning the performance of that predictive distribution.

\subsection{Exact sampling from the blockwise Stein predictive distribution}\label{Subsec:Exactsampling}

We provide an exact sampling method from the posterior distribution based on Stein prior $h^{(d)}(\theta^{(d)}):=||\theta^{(d)}||^{2-d}$
on $\mathbb{R}^{d}$.
Owing to the block structure,
it suffices to provide an exact sampling method from the posterior distribution based on Stein's prior.

We use the following mixture representation of Stein's prior:
\begin{align*}
	h^{(d)}(\theta^{(d)})=c(d)\int_{0}^{\infty}\frac{1}{(2\pi t)^{d/2}}\mathrm{e}^{-\frac{||\theta^{(d)}||^{2}}{2t}} \mathrm{d}t,
\end{align*}
where $c(d)$ is a constant depending only on $d$.
Thus, as for the posterior distribution of $h^{(d)}$, we have
\begin{align*}
	h^{(d)}(\theta^{(d)}|x^{(d)})=\int_{0}^{\infty}\pi(\theta^{(d)}|t,x^{(d)})f(t|x^{(d)})\mathrm{d}t,
\end{align*}
where 
\begin{align*}
	\pi(\theta^{(d)}|t,x^{(d)}):=\frac{1}{\left\{2\pi \left(\frac{\varepsilon^{2}}{1+\varepsilon^{2}/t}\right)\right\}^{d/2}}
		\exp\left\{-\frac{\sum_{i=1}^{d}(\theta_{i}-x_{i}/(1+\varepsilon^{2}/t))^{2}}{2\left(\frac{\varepsilon^{2}}{1+\varepsilon^{2}/t}\right)} \right\}
\end{align*}
and
\begin{align*}
	f(t|x^{(d)}):=\frac{\exp\left\{ -\frac{d}{2}\log(\varepsilon^{2}+t)-\frac{||x^{(d)}||^{2}}{2(\varepsilon^{2}+t)} \right\}}
	{\int_{0}^{\infty}\exp\left\{ -\frac{d}{2}\log(\varepsilon^{2}+\tilde{t})-\frac{||x^{(d)}||^{2}}{2(\varepsilon^{2}+\tilde{t})} \right\}\mathrm{d}\tilde{t}}.
\end{align*}
Here $\pi(\theta^{(d)}|t,x^{(d)})$ is the probability density function of the normal distribution.
Under the transformation $t\to \kappa:=\varepsilon^{2}/(\varepsilon^{2}+t)$,
the distribution $f(\kappa|x^{(d)})$ of $\kappa$ is a truncated Gamma distribution:
\begin{align*}
f(\kappa|x^{(d)})=1_{(0,1]}(\kappa)\frac{\kappa^{(d/2-1)-1}\exp\{-\frac{||x^{(d)}||^{2}}{2\varepsilon^{2}}\kappa\}}
{\int_{0}^{1}\tilde{\kappa}^{(d/2-1)-1}\exp\{-\frac{||x^{(d)}||^{2}}{2\varepsilon^{2}}\tilde{\kappa}\}\mathrm{d}\tilde{\kappa}}.
\end{align*}

Therefore,
we obtain an exact sampling from the posterior distribution based on Stein's prior
by sampling the normal distribution and the truncated Gamma distribution.
For the sampling from the truncated Gamma distribution,
we use the acceptance-rejection algorithm 
for truncated Gamma distributions
based on the mixture of beta distributions;
see \citet{Philippe(1997)}.

\subsection{Comparison with a fixed variance}\label{Subsec:Numerical}

Though we proved the asymptotic optimality of
the blockwise Stein predictive distribution with the WGB system,
it does not follow that the blockwise Stein predictive distribution
behaves well with a fixed variance $\varepsilon$.

In this subsection,
we examine the behavior with a fixed $\varepsilon$ of the blockwise Stein predictive distribution with the WGB system
compared to 
the plugin predictive distribution with the Bayes estimator based on the blockwise Stein prior
and
the asymptotically minimax predictive distribution in the Sobolev ellipsoid $\Theta_{\mathrm{Sobolev}}(\alpha,B)$ given in Theorem \ref{AsymptoticallyMinimaxTheorem}.
In this subsection,
we call the asymptotically minimax predictive distribution in the Sobolev ellipsoid $\Theta_{\mathrm{Sobolev}}(\alpha,B)$ given in Theorem \ref{AsymptoticallyMinimaxTheorem}
the Pinsker-type predictive distribution with $\alpha$ and $B$.

For the comparison,
we consider the 6 predictive settings with $\varepsilon=0.05$:
\begin{itemize}
\setlength{\leftskip}{1cm}
\item In the first setting,
$\theta_{i}=i^{-6}$ for $i\in\mathbb{N}$
and 
$\gamma=1$;
\item In the second setting,
$\theta_{i}=i^{-6}$ for $i\in\mathbb{N}$
and
$\gamma=1/3$;
\item In the third setting,
$\theta_{i}=i^{-6}$ for $i\in\mathbb{N}$
and
$\gamma=1/10$;
\item  In the fourth setting,
$\theta_{i}=i^{-1.5}$ for $i\in\mathbb{N}$
and 
$\gamma=1$;
\item In the fifth setting,
$\theta_{i}=i^{-1.5}$ for $i\in\mathbb{N}$
and 
$\gamma=1/3$;
\item In the sixth setting,
$\theta_{i}=i^{-1.5}$ for $i\in\mathbb{N}$
and 
$\gamma=1/10$.
\end{itemize}
Let $d(\varepsilon)=\lfloor 1/\varepsilon^{2}\rfloor=399$.

In each setting,
we obtain $1000$ samples of $y^{(d(\varepsilon))}$ distributed according to the blockwise Stein predictive distribution with the WGB system
up to the $d(\varepsilon)$-th order
using the sampling method described in Subsection \ref{Subsec:Exactsampling},
and we construct the coordinate-wise $80\%$-predictive interval of $y^{(d(\varepsilon))}$ using $1000$ samples.
In each setting,
we use the Pinsker-type predictive distribution with $\alpha$ and $B$
such that $\sum_{i=1}^{\infty}i^{2\alpha}\theta_{i}^{2}\leq B$:
we use $\alpha=2$ and $B=3$ in the firth, second, and third settings.
We use $\alpha=0.75$ and $B=3$ in the fourth, fifth, and sixth settings.

In each setting,
we obtain 5000 samples from the true distribution of $y$
and calculate the means of the coordinate-wise mean squared errors normalized by $\tilde{\varepsilon}^{2}$,
and then
calculate
the means and the standard deviations of the counts of the samples included in the predictive intervals.

\begin{table}[t]
\caption{Mean square error: The smallest value in each setting is underlined.}
	\begin{tabular}{|c|c|c|c|}\hline
	Setting number  & Bayes with WGBStein & Plugin with WGBStein & Pinsker \\ \hline
	First & 1.11             & 1.11                  & $\underline{1.03}$   \\\hline
	Second & 1.89             & 1.89                  & $\underline{1.12}$   \\\hline
	Third & 8.54             & 8.52                  & $\underline{1.85}$   \\\hline
	Fourth & 1.08             & 1.08                 & $\underline{1.04}$   \\\hline
	Fifth & 1.89             & 1.89                 & $\underline{1.50}$   \\\hline
	Sixth & 21.3             & 21.3                 & $\underline{11.0}$   \\\hline
\end{tabular}
\label{Table:MISE}

\end{table}
\begin{table}
\caption{Mean of the average percentages of the coverage (standard deviation): The value nearest to $80\%$ is underlined.}
	\begin{tabular}{|c|c|c|c|c|}\hline
	Setting number       & Bayes with WGBStein         & Plugin with WGBStein          & Pinsker \\ \hline
	First              & 82.4 (4.53$\times 10^{-2}$) & 78.6 (5.28$\times 10^{-2}$)   & $\underline{79.6}$ (4.03$\times 10^{-2}$) \\\hline
	Second              & 91.1 (8.08$\times 10^{-2}$) & 71.4 (14.9$\times 10^{-2}$)   & $\underline{79.7}$ (4.56$\times 10^{-2}$) \\\hline
	Third              & 98.9 (5.72$\times 10^{-2}$) & 45.8 (29.8$\times 10^{-2}$)   & $\underline{80.2}$ (5.06$\times 10^{-2}$) \\\hline
	Fourth              & 82.3 (3.04$\times 10^{-2}$) & 78.4 (6.11$\times 10^{-2}$)   & $\underline{79.8}$ (3.44$\times 10^{-2}$) \\\hline
	Fifth              & 91.6 (7.66$\times 10^{-2}$) & 71.0 (14.6$\times 10^{-2}$)   & $\underline{80.6}$ (8.21$\times 10^{-2}$) \\\hline
	Sixth              & 97.1 (13.9$\times 10^{-2}$) & 52.9 (26.9$\times 10^{-2}$)   & $\underline{83.5}$ (11.3$\times 10^{-2}$) \\\hline
\end{tabular}
\label{Table:APC}
\end{table}

Tables \ref{Table:MISE} and \ref{Table:APC} show
that the Pinsker-type predictive distribution (abbreviated by Pinsker) 
has the smallest mean squared error
and
has the sharpest predictive interval.
It is because the Pinsker-type predictive distribution uses $\alpha$ and $B$.
The blockwise Stein predictive distribution (abbreviated by Bayes with WGBStein)
and
the plugin predictive distribution with the Bayes estimator based on the blockwise Stein prior (abbreviated by Plugin with WGBStein)
have nearly the same performance in the mean squared error.
The blockwise Stein predictive distribution
has a wider predictive interval than the plugin predictive distribution.
Its predictive interval has a smaller variance
than that of the plugin predictive distribution in all settings.
In the next paragraph,
we consider the reason for this phenomenon by using
the transformation of the infinite sequence model to the function model discussed in Section 2.

\begin{figure}[h]
	\begin{minipage}{0.4\hsize}
		{\centering
			\subfigure[The true function and Pinsker in the second setting]{\includegraphics[width=4cm]{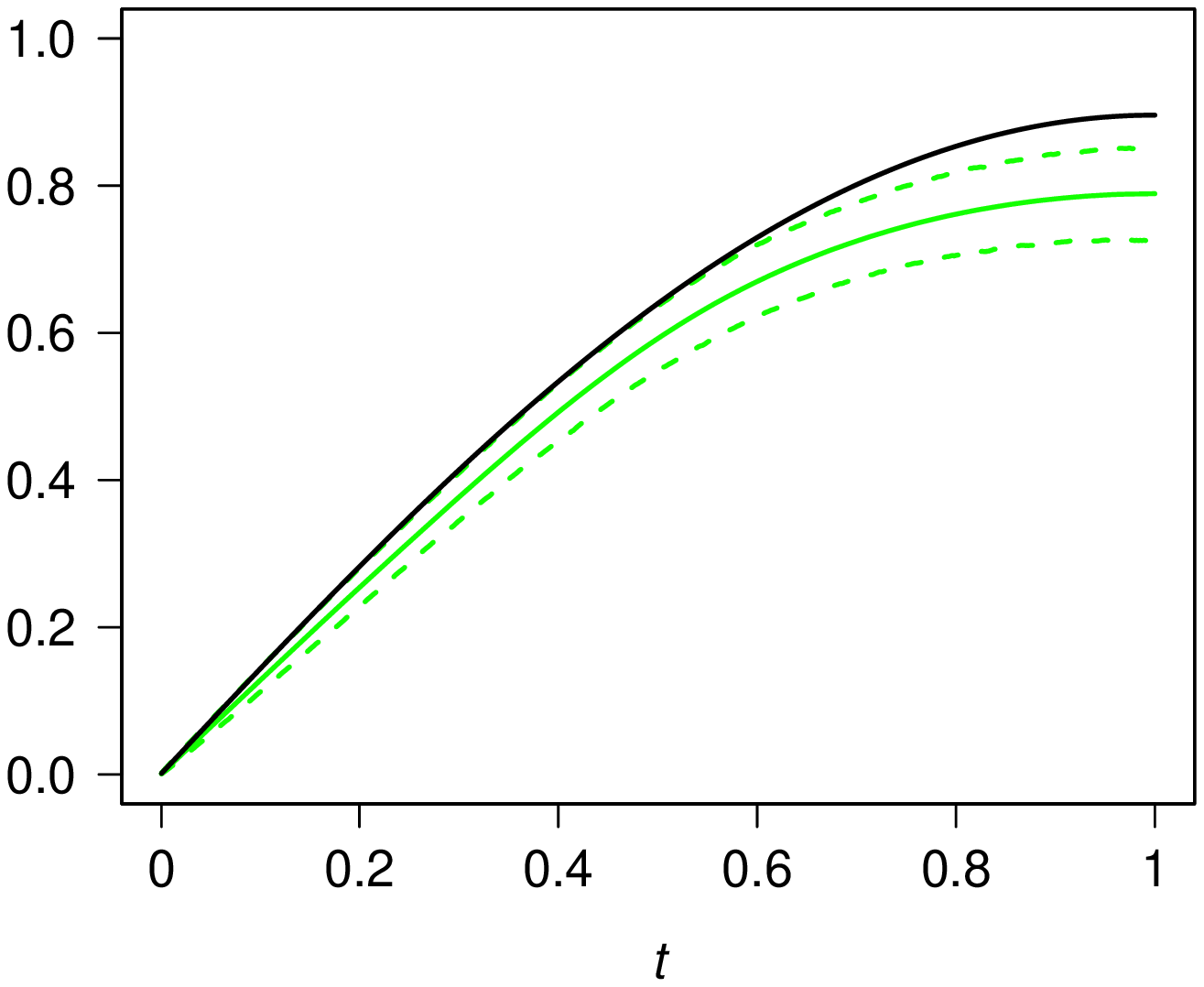}}
		\label{fig:PI:1st}
	}
	\end{minipage}
	\begin{minipage}{0.4\hsize}
		{\centering
			\subfigure[The Bayes and the plugin with WGBStein in the second setting]{\includegraphics[width=4cm]{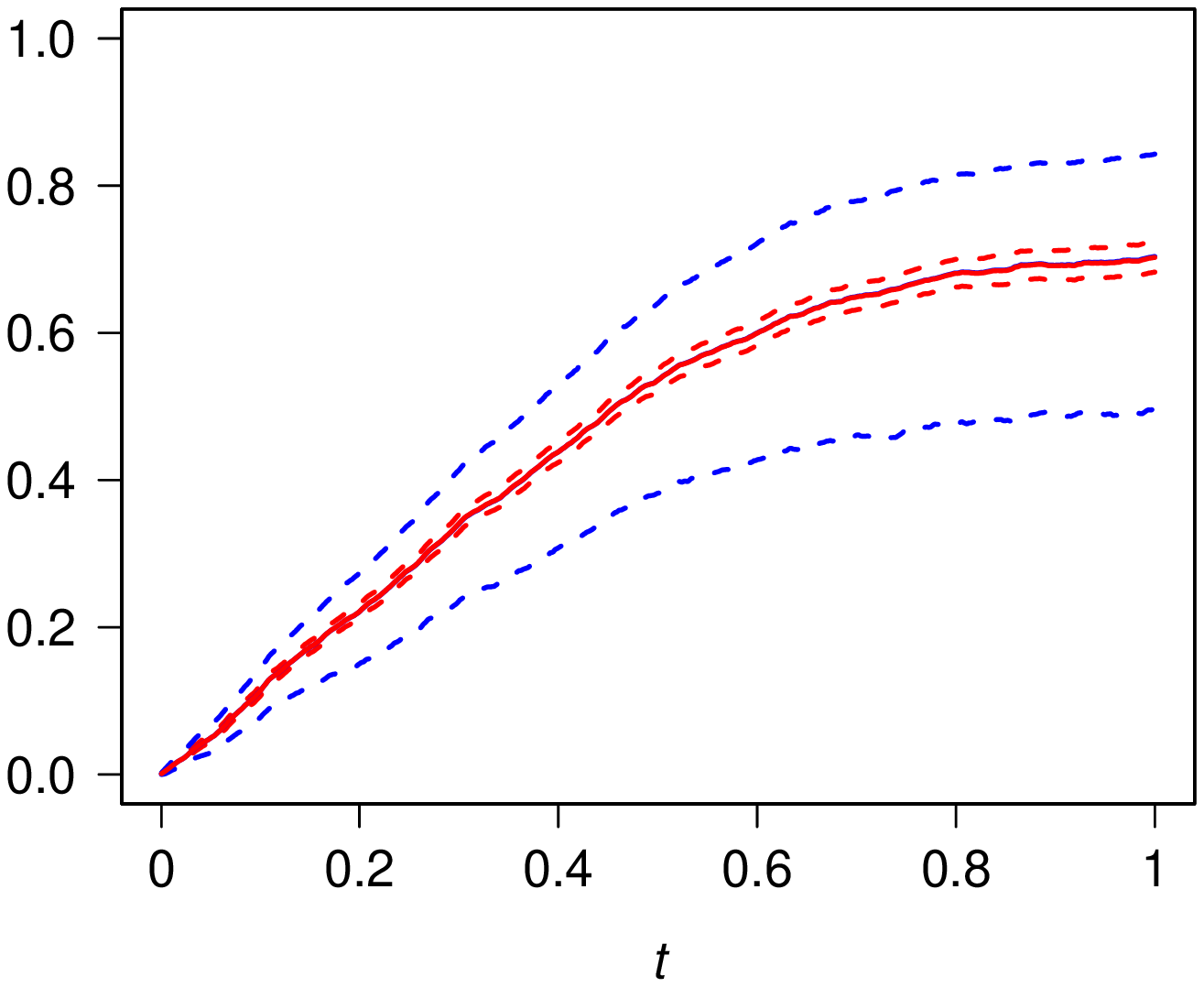}}
		\label{fig:PI:2nd}	
	}
	\end{minipage}
	\begin{minipage}{0.4\hsize}
		{\centering
			\subfigure[The true function and Pinsker in the fourth setting]{\includegraphics[width=4cm]{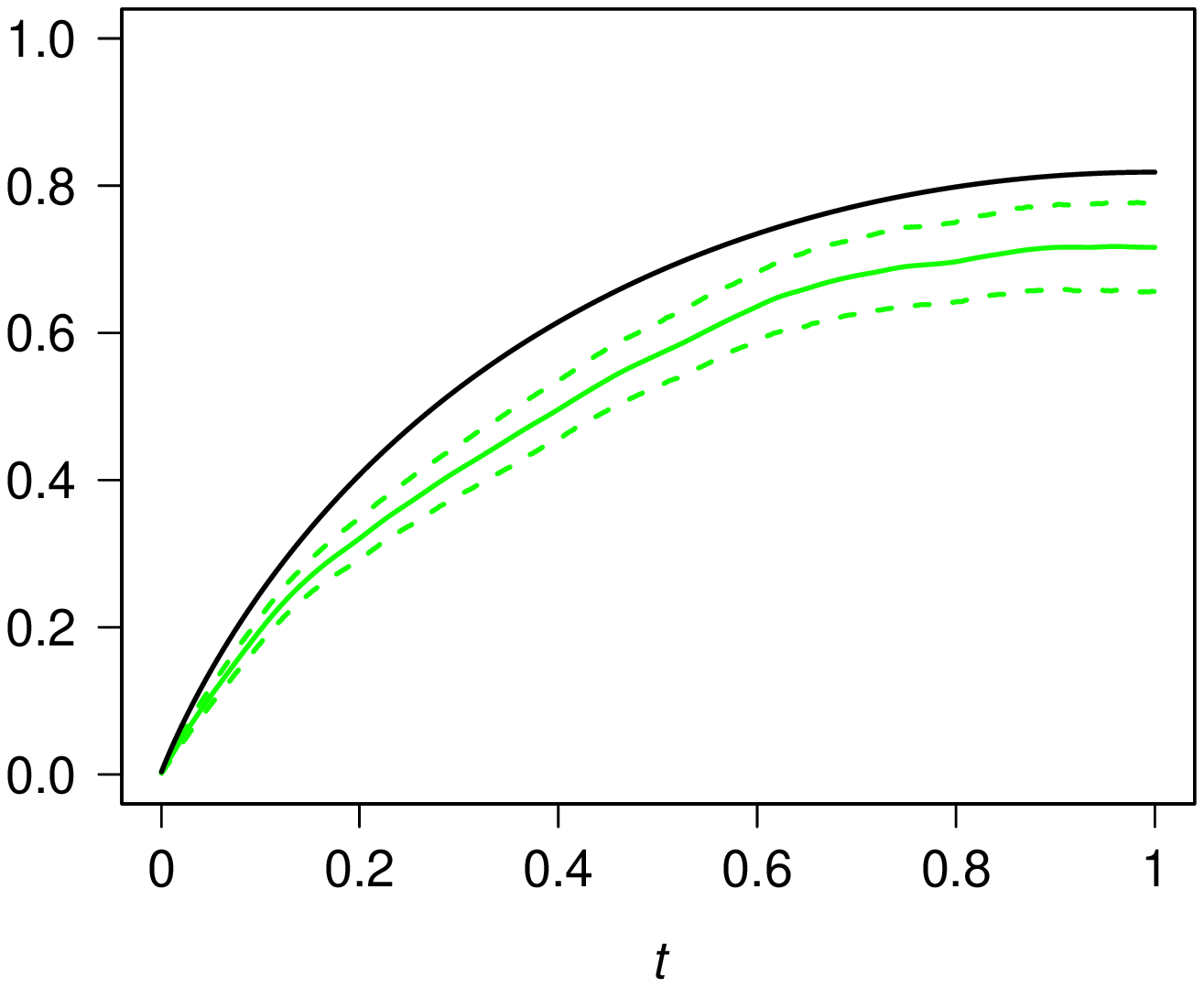}}
		\label{fig:PI:3rd}
	}
	\end{minipage}
	\begin{minipage}{0.4\hsize}
		{\centering
			\subfigure[The Bayes and the plugin with WGBStein in the fourth setting]{\includegraphics[width=4cm]{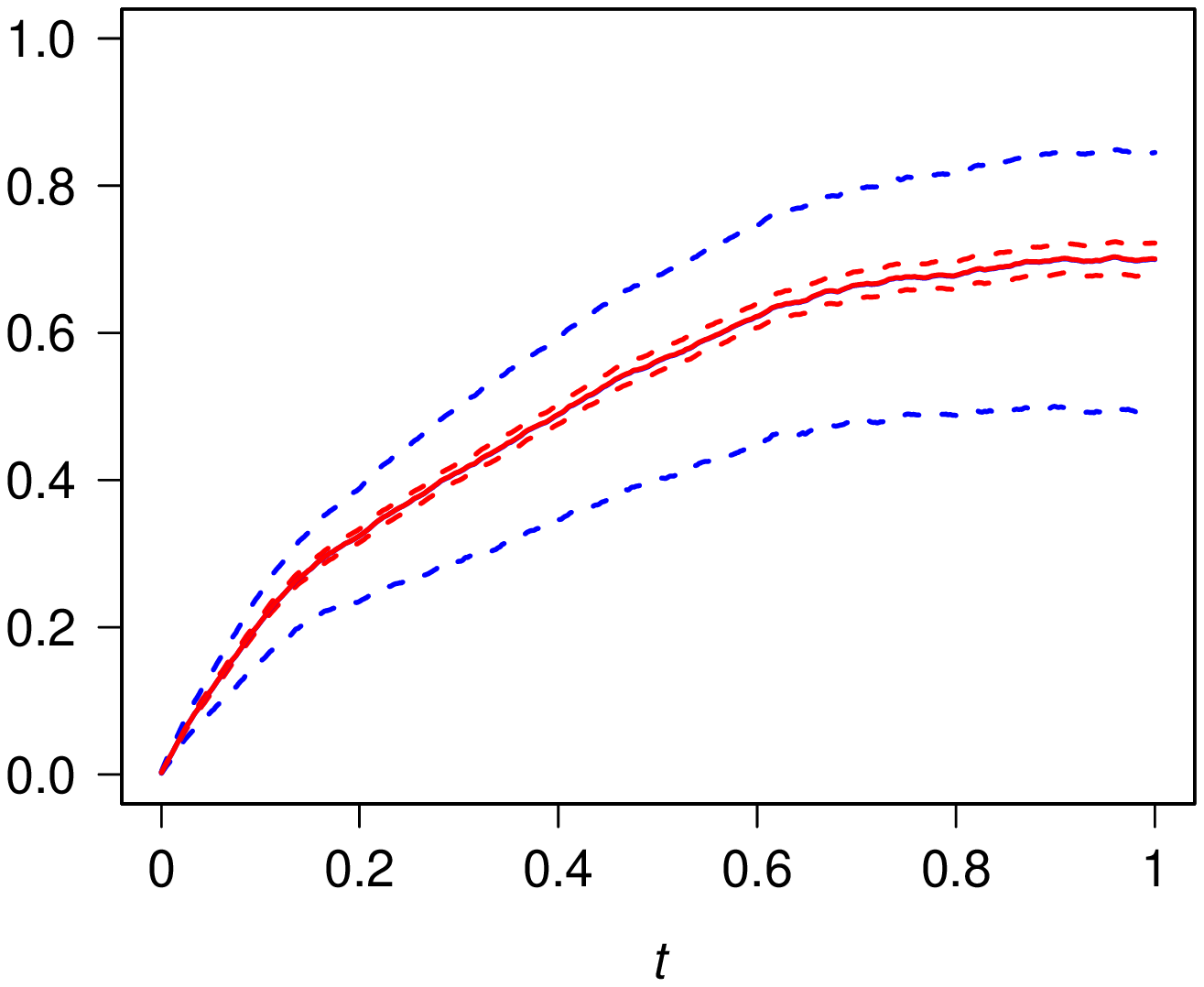}}
			\label{fig:PI:4th}
		}
		\end{minipage}
\caption{
The true function corrsponding to $\theta$ (solid black line),
the mean paths (solid lines) and the pointwise $80\%$ predictive intervals (dashed lines) of the predictive distributions.
}
\label{fig:PI}
\end{figure}

Using the function representation of the infinite sequence model discussed in Section \ref{Sec:Equivalence},
we examine the behavior of the predictive distributions in the second and fifth settings.
Figure \ref{fig:PI} shows the mean path and the predictive intervals of predictive distributions at $t\in\{i/1000\}_{i=1}^{1000}$
and the values of the true function at $t\in\{i/1000\}_{i=1}^{1000}$.
Figure \ref{fig:PI} (a), Figure \ref{fig:PI} (b), Figure \ref{fig:PI} (c), and Figure \ref{fig:PI} (d)
represent 
the Pinsker-type predictive distribution and the true function in the second setting,
the blockwise Stein predictive distribution and the plugin predictive distribution in the second setting,
the Pinsker-type predictive distribution and the true function in the fourth setting,
and
the blockwise Stein predictive distribution and the plugin predictive distribution in the fourth setting,
respectively.
The solid line represents the true function and the mean paths.
The dashed line represents the pointwise $80\%$ predictive intervals.
The black, green, blue, and red lines correspond to
the true function, the Pinsker-type predictive distribution, the blockwise Stein predictive distribution,
and
the plugin predictive distribution,
respectively.

The mean paths of the blockwise Stein predictive distribution and the plugin predictive distributions
are more distant from the true function than that of the Pinsker-type predictive distribution,
corresponding to the results in Table \ref{Table:MISE}.
The predictive intervals of the blockwise Stein predictive distribution are wider
than these of the other predictive distributions,
corresponding to the results in Table \ref{Table:APC}.
Though the blockwise Stein predictive distribution has a mean path that is more distant from the true function than the Pinsker-type predictive distribution,
it has a wider predictive interval and captures future observations.
In contrast,
although
the plugin predictive distribution has nearly the same mean path as the blockwise Stein predictive distribution does,
it has a narrow predictive interval and does not capture future observations.

\section{Discussions and Conclusions}\label{Sec:Conclusion}

In the paper,
we have considered asymptotically minimax Bayesian predictive distributions in an infinite sequence model.
First,
we have provided the connection between prediction in a function model and prediction in an infinite sequence model.
Second,
we have constructed an asymptotically minimax Bayesian predictive distribution
for the setting in which the parameter space is a known ellipsoid.
Third,
using the product of Stein's priors based on the division of the parameter into blocks,
we have constructed an asymptotically minimax adaptive Bayesian predictive distribution in the family of Sobolev ellipsoids.

We established the fundamental results of prediction in the infinite-dimensional model using the asymptotics as $\varepsilon\to 0$.
The approach was motivated by \citet{XuandLiang(2010)}.
Since it is not always appropriate to use asymptotics in applications,
the next step is to provide the result for a fixed $\varepsilon$.

We discussed the asymptotic minimaxity and the adaptivity for the ellipsoidal parameter space.
There are many other types of parameter space in high-dimensional and nonparametric models;
for example, \citet{MukherjeeandJohnstone(2015)} discussed the asymptotically minimax prediction in high-dimensional Gaussian sequence model under sparsity.
For future work,
we should focus on the asymptotically minimax adaptive predictive distributions in other parameter spaces.

\section{Acknowledgements}
The authors thank the Editor, an associate editor, reviewers for their careful reading and constructive suggestions on the manuscript.
This work is supported by JSPS KAKENHI Grand number 26280005.

\appendix
\section{Proofs of Lemmas in Section 2}\label{Appendix:ProofofSection2}

\begin{proof}[Proof of Lemma \ref{GaussianMeasures}]
The proof is similar to that of Lemmas 5.1 and 6.1 in \citet{BelitserandGhosal(2003)}.
We denote the expectation of $X$ and $Y$ with respect to $P_{\theta}$ and $Q_{\theta}$ by $\mathrm{E}_{X,Y|\theta}$.

First,
we show that $Q_{\mathrm{G}_{\tau}}$ and $Q_{\theta}$ are mutually absolutely continuous given $X=x$ $P_{\theta}$-a.s.~if $\theta\in l_{2}$ and $\tau\in l_{2}$.
From Kakutani's theorem (pp.~150--151 in \citet{Williams(1991)}),
$Q_{\mathrm{G}_{\tau}}$ and $Q_{\theta}$ are mutually absolutely continuous given $X=x$ $P_{\theta}$-a.s.~if and only if
\begin{eqnarray}
	0 < \mathop{\mathlarger{\prod}}_{i=1}^{\infty} \mathlarger{\int}
	\sqrt{\frac{\mathrm{d}\mathcal{N}\left(\frac{1/\varepsilon^{2}}{1/\varepsilon^{2}+1/\tau^{2}_{i}}x_{i},
\frac{1}{1/\varepsilon^{2}+1/\tau^{2}_{i}}+\tilde{\varepsilon}^{2}
\right)}{\mathrm{d}y_{i}}\frac{\mathrm{d}\mathcal{N}\left(\theta_{i},\tilde{\varepsilon}^{2}\right)}{\mathrm{d}y_{i}}}
\mathrm{d}y_{i} \text{ $P_{\theta}$-a.s.}.
\label{Affinity}
\end{eqnarray}
Since the right hand side of (\ref{Affinity})
is the limit of the product of
\begin{eqnarray}
	\mathop{\mathlarger{\prod}}_{i=1}^{p} 
	\left\{1-\frac{(1/(1/\varepsilon^{2} +1/\tau_{i}^{2}))^{2}}{2\tilde{\varepsilon}^{2}+(1/(1/\varepsilon^{2} +1/\tau_{i}^{2}))}\right\}^{1/2}
\label{diverge_1}
\end{eqnarray}
and
\begin{eqnarray}
	\mathop{\mathlarger{\prod}}_{i=1}^{p} 
\mathrm{exp}\left\{-\frac{\left(\theta_{i}-\frac{1/\varepsilon^{2}}{1/\varepsilon^{2}+1/\tau_{i}^{2}}x_{i}\right)^{2}}{4\left(2\tilde{\varepsilon}^{2}+\frac{1}{1/\varepsilon^{2}+1/\tau_{i}^{2}}\right)}\right\}
\label{diverge_2}
\end{eqnarray}
as $p\to\infty$,
it suffices to show that both (\ref{diverge_1}) and (\ref{diverge_2}) converge to non-zero constants.
Quantity (\ref{diverge_1}) converges to a non-zero constant 
because
the product $\prod_{i=1}(1-a_{i})$ with $\{a_{i}>0\}_{i=1}^{\infty}$ converges to a non-zero constant
provided that $\sum_{i=1}a_{i}$ converges,
and because
\begin{eqnarray*}
	\mathlarger{\sum}_{i=1}^{\infty}\frac{(1/(1/\varepsilon^{2} +1/\tau_{i}^{2}))^{2}}{2\tilde{\varepsilon}^{2}+(1/(1/\varepsilon^{2} +1/\tau_{i}^{2}))}
	\leq \sum_{i=1}^{\infty} \frac{1}{1/\varepsilon^{2}+1/\tau_{i}^{2}}
	\leq \sum_{i=1}^{\infty} \tau_{i}^{2}<\infty.
\end{eqnarray*}
Quantity (\ref{diverge_2}) converges to a non-zero constant
because
\begin{align*}
\frac{1}{4}\sum_{i=1}^{\infty}
	\left\{	\frac{\left(\theta_{i}-\frac{1/\varepsilon^{2}}{1/\varepsilon^{2}+1/\tau_{i}^{2}}x_{i}\right)^{2}}{(2\tilde{\varepsilon}^{2}+1/(1/\varepsilon^{2}+1/\tau_{i}^{2}))} \right\}
	&\leq
	\frac{1}{8\tilde{\varepsilon}^{2}}\sum_{i=1}^{\infty}\theta_{i}^{2}
	+
	\frac{1}{8\tilde{\varepsilon}^{2}\varepsilon^{2}}
	\sum_{i=1}^{\infty}\tau_{i}^{2}(x_{i}-\theta_{i})^{2}
	\nonumber\\
	&+
	\frac{1}{4\tilde{\varepsilon}^{2}\varepsilon^{2}}
	\left(\sum_{i=1}^{\infty}\theta_{i}^{2}\right)^{1/2}
	\left(\sum_{i=1}^{\infty}\tau_{i}^{2}(x_{i}-\theta_{i})^{2}\right)^{1/2}.
\end{align*}
Thus, $Q_{\mathrm{G}_{\tau}}$ and $Q_{\theta}$ are mutually absolutely continuous given $X=x$ $P_{\theta}$-a.s..

\if(0)
For the convergence of (\ref{diverge_2}),
consider the following decomposition of the minus log of (\ref{diverge_2}):
\begin{eqnarray}
	\frac{1}{4}\sum_{i=1}^{\infty}
	\left\{	\frac{\left(\theta_{i}-\frac{1/\varepsilon^{2}}{1/\varepsilon^{2}+1/\tau_{i}^{2}}x_{i}\right)^{2}}{(2\tilde{\varepsilon}^{2}+1/(1/\varepsilon^{2}+1/\tau_{i}^{2}))} \right\}
	&=&
	\frac{1}{4}\sum_{i=1}^{\infty}\frac{\left(\frac{1/\tau_{i}^{2}}{1/\varepsilon^{2}+1/\tau_{i}^{2}}\right)^{2}\theta_{i}^{2}}{2\tilde{\varepsilon}^{2}+1/(1/\varepsilon^{2}+1/\tau_{i}^{2})}
	+
	\frac{1}{4}\sum_{i=1}^{\infty}\frac{\left(\frac{1/\varepsilon^{2}}{1/\varepsilon^{2}+1/\tau_{i}^{2}}\right)^{2}(x_{i}-\theta_{i})^{2}}{2\tilde{\varepsilon}^{2}+1/(1/\varepsilon^{2}+1/\tau_{i}^{2})}
	\nonumber\\
	&&\quad
	-\frac{1}{2}\sum_{i=1}^{\infty}\frac{\frac{1/\tau_{i}^{2}}{1/\varepsilon^{2}+1/\tau_{i}^{2}}\frac{1/\varepsilon^{2}}{1/\varepsilon^{2}+1/\tau_{i}^{2}}\theta_{i}(x_{i}-\theta_{i})}{2\tilde{\varepsilon}^{2}+1/(1/\varepsilon^{2}+1/\tau_{i}^{2})}.
	\label{decomp_diverge_2}
\end{eqnarray}
Note that 
\begin{eqnarray}
	1/(b+c)<\min\{1/b,1/c\} \text{ for positive $b$ and positive $c$}
	\label{posiposi}
\end{eqnarray}
and
\begin{eqnarray}
	b/(b+c)<1 \text{ for positive $b$ and positive $c$}.
	\label{posione}
\end{eqnarray}
From inequality (\ref{posiposi}),
the first term of the right hand side in (\ref{decomp_diverge_2}) is bounded 
above by $(1/8\tilde{\varepsilon}^{2})\sum_{i}\theta_{i}^{2}$.
From inequalities (\ref{posiposi}) and (\ref{posione}),
the second term of the right hand side in (\ref{decomp_diverge_2}) is bounded above by 
$(1/8\tilde{\varepsilon}^{2}\varepsilon^{2})\sum_{i=1}^{\infty}\tau_{i}^{2}(x_{i}-\theta_{i})^{2}$.
From the Cauchy-Schwartz inequality and 
from the inequalities that for $i\in\mathbb{N}$
\begin{eqnarray*}
	\frac{\left(\frac{1/\varepsilon^{2}}{1/\varepsilon^{2}+1/\tau_{i}^{2}}\right)^{2}}{2\tilde{\varepsilon}^{2}+1/(1/\varepsilon^{2}+1/\tau_{i}^{2})}
	\leq \frac{1}{2\tilde{\varepsilon}^{2}}
	&\text{ and }&
	\frac{\left(\frac{1/\varepsilon^{2}}{1/\varepsilon^{2}+1/\tau_{i}^{2}}\right)^{2}}{2\tilde{\varepsilon}^{2}+1/(1/\varepsilon^{2}+1/\tau_{i}^{2})}
	\leq \frac{\tau_{i}^{2}}{2\tilde{\varepsilon}^{2}\varepsilon^{2}},
\end{eqnarray*}
the third term of the right hand side in (\ref{decomp_diverge_2}) is bounded above as follows:
\begin{align*}
	-&\frac{1}{2}\mathlarger{\sum}_{i=1}^{\infty}\frac{\frac{1/\tau_{i}^{2}}{1/\varepsilon^{2}+1/\tau_{i}^{2}}\frac{1/\varepsilon^{2}}{1/\varepsilon^{2}+1/\tau_{i}^{2}}\theta_{i}(x_{i}-\theta_{i})}{2\tilde{\varepsilon}^{2}+1/(1/\varepsilon^{2}+1/\tau_{i}^{2})}
	\nonumber\\
	&\leq
	\frac{1}{2}\left(\mathlarger{\sum}_{i=1}^{\infty}\frac{\left(\frac{1/\tau_{i}^{2}}{1/\varepsilon^{2}+1/\tau_{i}^{2}}\right)^{2}\theta_{i}^{2}}{2\tilde{\varepsilon}^{2}+1/(1/\varepsilon^{2}+1/\tau_{i}^{2})}\right)^{1/2}
	\left(\mathlarger{\sum}_{i=1}^{\infty}\frac{\left(\frac{1/\varepsilon^{2}}{1/\varepsilon^{2}+1/\tau_{i}^{2}}\right)^{2}(x_{i}-\theta_{i})^{2}}{2\tilde{\varepsilon}^{2}+1/(1/\varepsilon^{2}+1/\tau_{i}^{2})} \right)^{1/2}
	\nonumber\\
	&\leq
	\frac{1}{4}\frac{1}{\tilde{\varepsilon}^{2}\varepsilon}\left(\sum_{i=1}^{\infty}\theta_{i}^{2}\right)^{1/2}\left(\sum_{i=1}^{\infty}\tau_{i}^{2}(x_{i}-\theta_{i})^{2}\right)^{1/2}.
\end{align*}
Since $\theta\in l_{2}$, the first term in (\ref{decomp_diverge_2}) is finite.
To show that the second and the third terms are finite $P_{\theta}$-a.s.,
let $M_{n}:=\sum_{i=1}^{n}[ \tau_{i}^{2}(X_{i}-\theta_{i})^{2}-\tau_{i}^{2}]$.
Here, $M_{n}$ is a zero-mean martingale 
and since $\tau\in l_{2}$,
$M_{n}$ satisfies
$\mathop{\sup}_{n}\mathrm{E}_{X,Y|\theta}[|M_{n}|]\leq 2\sum_{i=1}^{\infty}\tau_{i}^{2}<\infty$.
From the martingale convergence theorem,
$M_{\infty}:=\lim M_{n}=\sum_{i=1}^{\infty}[ \tau_{i}^{2}(X_{i}-\theta_{i})^{2}-\tau_{i}^{2}]$ 
exists and is finite $P_{\theta}$-a.s..
Then,
the second term and the third term in (\ref{decomp_diverge_2}) are finite $P_{\theta}$-a.s..
Thus, $Q_{G_{\tau}}$ and $Q_{\theta}$ are mutually absolutely continuous given $X=x$ $P_{\theta}$-a.s..
\fi

Second,
we show that $R(\theta,Q_{\mathrm{G}_{\tau}})$ is given by (\ref{KL_Bayes_productNormal}).
Let $\pi^{(d)}:\mathbb{R}^{\infty}\to\mathbb{R}^{d}$ be the finite dimensional projection $\pi^{(d)}(x)=(x_{1},\ldots,x_{d})$.
Let $Q^{(d)}_{\theta^{(d)}}$ and $Q^{(d)}_{\mathrm{G}_{\tau}}$ be the induced probability measures of $Q_{\theta}$ and $Q_{\mathrm{G}_{\tau}}$
by the finite-dimensional projection $\pi^{(d)}$, respectively.
From Kakutani's theorem, we have
\begin{eqnarray*}
	R(\theta,Q_{\mathrm{G}_{\tau}})=\mathrm{E}_{X,Y|\theta}[\lim \log(\mathrm{d}Q^{(d)}_{\theta^{(d)}}/\mathrm{d}Q^{(d)}_{\mathrm{G}_{\tau}})].
\end{eqnarray*}
Then,
it suffices to show that the almost sure convergence in the right hand side
is also the convergence in the expectation.
Here,
\begin{eqnarray*}
	\log(\mathrm{d}Q^{(d)}_{\theta^{(d)}}/\mathrm{d}Q^{(d)}_{\mathrm{G}_{\tau}})=S_{d}
+\mathop{\sum}_{i=1}^{d}
\left[
\frac{1}{2}\frac{v^{2}_{\varepsilon,\tilde{\varepsilon}}+\theta^{2}_{i}}{v^{2}_{\varepsilon,\tilde{\varepsilon}}+\tau^{2}_{i}}
-\frac{1}{2}\frac{v^{2}_{\varepsilon}+\theta^{2}_{i}}{v^{2}_{\varepsilon}+\tau^{2}_{i}}
+\frac{1}{2}\log\left(\frac{1+\tau_{i}^{2}/v^{2}_{\varepsilon,\tilde{\varepsilon}}}{1+\tau_{i}^{2}/v^{2}_{\varepsilon}}\right)
\right],
\end{eqnarray*}
where $S_{d}$ is
\begin{eqnarray*}
S_{d}:=\mathop{\sum}_{i=1}^{d}\left[
-\frac{(Y_{i}-\theta_{i})^{2}}{2\tilde{\varepsilon}^{2}}+\frac{\left(Y_{i}-\frac{1/\varepsilon^{2}}{1/\varepsilon^{2}+1/\tau_{i}^{2}}X_{i}\right)^{2}}{2(\tilde{\varepsilon}^{2}+1/(1/\varepsilon^{2}+1/\tau_{i}^{2}))}
-\frac{1}{2}\frac{v^{2}_{\varepsilon,\tilde{\varepsilon}}+\theta^{2}_{i}}{v^{2}_{\varepsilon,\tilde{\varepsilon}}+\tau^{2}_{i}}
	+\frac{1}{2}\frac{v^{2}_{\varepsilon}+\theta^{2}_{i}}{v^{2}_{\varepsilon}+\tau^{2}_{i}}
\right].
\end{eqnarray*}
Here $S_{d}$ is a zero-mean martingale such that $\mathop{\sup}_{d}\mathrm{E}_{X,Y|\theta}[S_{d}^{2}]<\infty$;
From the martingale convergence theorem (p. 111 in \citet{Williams(1991)}), 
$S_{d}$ converges to $S_{\infty}:=\lim S_{d}$, $P_{\theta}$-a.s.~and $\mathrm{E}_{X,Y|\theta}[S_{d}-S_{\infty}]^{2}\to 0$.
Since $\mathrm{E}_{X,Y|\theta}[S_{d}]\to\mathrm{E}_{X,Y|\theta}[S_{\infty}]=0$,
equality (\ref{KL_Bayes_productNormal}) follows.
\end{proof}

\if(0)
\begin{proof}[Proof of the properties of $S_{n}$ in Lemma 2.1]
It suffices to show that 
$\mathrm{E}_{X,Y|\theta}[S_{n}-S_{n-1}]=0$ for $n\in\mathbb{N}$
and
$\sum_{n=1}^{\infty}\mathrm{E}_{X,Y|\theta}[S_{n}-S_{n-1}]^{2}<\infty$.
For every $n\in\mathbb{N}$,
$S_{n}-S_{n-1}$ is given by
\begin{align*}
S_{n}-S_{n-1}
=&
-\frac{1}{2\tilde{\varepsilon}^{4}}
\frac{\tau_{n}^{2}v^{2}_{\varepsilon,\tilde{\varepsilon}}}{\tau_{n}^{2}+v^{2}_{\varepsilon,\tilde{\varepsilon}}}
(Y_{n}-\theta_{n})^{2}
+\frac{1}{2\tilde{\varepsilon}^{2}}
\frac{v^{2}_{\varepsilon,\tilde{\varepsilon}}}{v^{2}_{\varepsilon}}
\frac{\tau_{n}^{2}+v^{2}_{\varepsilon}}{\tau_{n}^{2}+v^{2}_{\varepsilon,\tilde{\varepsilon}}}
\left(
\frac{\tau_{n}^{2}}{\tau_{n}^{2}+v^{2}_{\varepsilon}}
\right)^{2}
(X_{n}-\theta_{n})^{2}
\nonumber\\
&+
\frac{1}{\tilde{\varepsilon}^{2}}
\frac{v^{2}_{\varepsilon,\tilde{\varepsilon}}}{\tau_{n}^{2}+v^{2}_{\varepsilon,\tilde{\varepsilon}}}
\theta_{n}(Y_{n}-\theta_{n})
-
\frac{1}{\tilde{\varepsilon}^{2}}
\frac{v^{2}_{\varepsilon,\tilde{\varepsilon}}}{\tau_{n}^{2}+v^{2}_{\varepsilon,\tilde{\varepsilon}}}
\frac{\tau_{n}^{2}}{\tau_{n}^{2}+v^{2}_{\varepsilon}}\theta_{n}(X_{n}-\theta_{n})
\nonumber\\
&-
\frac{1}{\tilde{\varepsilon}^{2}\varepsilon^{2}}
\frac{v^{2}_{\varepsilon,\tilde{\varepsilon}}}{\tau_{n}^{2}+v^{2}_{\varepsilon,\tilde{\varepsilon}}}
\tau_{n}^{2}(X_{n}-\theta_{n})(Y_{n}-\theta_{n})
+
\frac{1}{2}\left(
	\frac{1}{v^{2}_{\varepsilon,\tilde{\varepsilon}}+\tau_{n}^{2}}
	-
	\frac{1}{v^{2}_{\varepsilon}+\tau_{n}^{2}}
\right)\tau_{n}^{2}.
\end{align*}
Then, $\mathrm{E}_{X,Y|\theta}[S_{n}-S_{n-1}]=0$.
Since $X_{n}-\theta_{n}$ and $Y_{n}-\theta_{n}$ are independent
and $X_{n}-\theta$ and $Y_{n}-\theta_{n}$ are distributed according to $\mathcal{N}(0,\varepsilon^{2})$
and $\mathcal{N}(0,\tilde{\varepsilon}^{2})$,respectively,
\begin{align*}
	\mathrm{E}_{X,Y|\theta}&[S_{n}-S_{n-1}]^{2}
	\nonumber\\
	=&
	\frac{3}{4}\frac{v^{4}_{\varepsilon,\tilde{\varepsilon}}}{\tilde{\varepsilon}^{4}}
	\left(\frac{\tau_{n}^{2}}{\tau_{n}^{2}+v^{2}_{\varepsilon,\tilde{\varepsilon}}}\right)^{2}
	+
	\frac{3}{4}\frac{\varepsilon^{4}}{\tilde{\varepsilon}^{4}}
	\left(\frac{v^{2}_{\varepsilon,\tilde{\varepsilon}}}{v^{2}_{\varepsilon}}\right)^{2}
	\left(\frac{\tau_{n}^{2}+v^{2}_{\varepsilon}}{\tau_{n}^{2}+v^{2}_{\varepsilon,\tilde{\varepsilon}}}
	\right)^{2}
	\left(
		\frac{\tau_{n}^{2}}{\tau_{n}^{2}+v^{2}_{\varepsilon}}
	\right)^{4}
	\nonumber\\
	&+
	\frac{1}{\tilde{\varepsilon}^{2}}
	\left(\frac{v^{2}_{\varepsilon,\tilde{\varepsilon}}}{\tau_{n}^{2}+v^{2}_{\varepsilon,\tilde{\varepsilon}}}\right)^{2}
	\theta_{n}^{2}
	+
	\frac{\varepsilon^{2}}{\tilde{\varepsilon}^{4}}
	\left(
		\frac{v^{2}_{\varepsilon,\tilde{\varepsilon}}}{\tau_{n}^{2}+v^{2}_{\varepsilon,\tilde{\varepsilon}}}
	\right)^{2}
	\left(
		\frac{\tau_{n}^{2}}{\tau_{n}^{2}+v^{2}_{\varepsilon}}
	\right)^{2}
	\theta_{n}^{2}
	\nonumber\\
	&+
	\frac{1}{\varepsilon^{2}\tilde{\varepsilon}^{2}}
	\left(
		\frac{v^{2}_{\varepsilon,\tilde{\varepsilon}}\tau_{n}^{2}}{\tau_{n}^{2}+v^{2}_{\varepsilon,\tilde{\varepsilon}}}
	\right)^{2}
	+
	\frac{1}{4}
	\left(\frac{v^{2}_{\varepsilon}-v^{2}_{\varepsilon,\tilde{\varepsilon}}}{(\tau_{n}^{2}+v^{2}_{\varepsilon})(\tau_{n}^{2}+v^{2}_{\varepsilon,\tilde{\varepsilon}})}\right)^{2}
	\tau_{n}^{4}
	\nonumber\\
	&+
	\frac{\varepsilon^{2}}{2\tilde{\varepsilon}^{2}}
	\frac{v^{2}_{\varepsilon}-v^{2}_{\varepsilon,\tilde{\varepsilon}}}{(\tau_{n}^{2}+v^{2}_{\varepsilon})(\tau_{n}^{2}+v^{2}_{\varepsilon,\tilde{\varepsilon}})}
	\frac{v^{2}_{\varepsilon,\tilde{\varepsilon}}}{v^{2}_{\varepsilon}}
	\frac{\tau_{n}^{2}+v^{2}_{\varepsilon}}{\tau_{n}^{2}+v^{2}_{\varepsilon,\tilde{\varepsilon}}}
	\left(
		\frac{\tau_{n}^{2}}{\tau_{n}^{2}+v^{2}_{\varepsilon}}
	\right)^{2}
	\tau_{n}^{2}
	+NT,
\end{align*}
where $NT$ is the negative terms.
By omitting the negative terms and 
from inequalities (\ref{posiposi}) and (\ref{posione}),
\begin{eqnarray*}
\mathrm{E}_{X,Y|\theta}[S_{n}-S_{n-1}]^{2}
&\leq&
\left[\frac{3}{4}\frac{v^{2}_{\varepsilon,\tilde{\varepsilon}}}{\tilde{\varepsilon}^{4}}
	+\frac{3}{4}\frac{v^{2}_{\varepsilon,\tilde{\varepsilon}}}{\tilde{\varepsilon}^{4}}
	+\frac{v^{2}_{\varepsilon,\tilde{\varepsilon}}}{\tilde{\varepsilon}^{2}\varepsilon^{2}}
	+\frac{1}{4}\frac{(v^{2}_{\varepsilon}-v^{2}_{\varepsilon,\tilde{\varepsilon}})^{2}}{\varepsilon^{4}v^{2}_{\varepsilon,\tilde{\varepsilon}}}
	+\frac{1}{2}\frac{(v^{2}_{\varepsilon}-v^{2}_{\varepsilon,\tilde{\varepsilon}})}{\tilde{\varepsilon}^{2}v^{2}_{\varepsilon,\tilde{\varepsilon}}}
\right]\tau_{n}^{2}
+\left[\frac{1}{\tilde{\varepsilon}^{2}}+\frac{\varepsilon^{2}}{\tilde{\varepsilon}^{4}}\right]\theta_{n}^{2}.
\end{eqnarray*}
Thus, by combining the above inequality with the assumption that $\theta\in l_{2}$ and $\tau\in l_{2}$,
we obtain $\sum_{n=1}^{\infty}\mathrm{E}_{X,Y|\theta}[S_{n}-S_{n-1}]^{2}<\infty$.
\end{proof}
\fi

\vspace{4mm}

\begin{proof}[Proof of Lemma \ref{LinearMinimax}]

First,
the finiteness of $T(\varepsilon,\tilde{\varepsilon})$ is derived from the assumption for $a=(a_{1},a_{2},\ldots)$.
$\lambda(\varepsilon,\tilde{\varepsilon})$ is uniquely determined
because
$\sum_{i=1}^{\infty}a_{i}^{2}\left(\tau^{*}_{i}(\varepsilon,\tilde{\varepsilon})\right)^{2}$
is continuous and strictly decreasing with respect to $\lambda$.
Note that $\tau^{*}(\varepsilon,\tilde{\varepsilon})$ is also a function with respect to $\lambda(\varepsilon,\tilde{\varepsilon})$.

Second, 
we show that
\begin{align*}
\mathop{\sup}_{\theta\in\Theta(a,B)}R(\theta,Q_{\mathrm{G}_{\tau=\tau^{*}(\varepsilon,\tilde{\varepsilon})}})
\geq\mathop{\sup}_{\theta\in\Theta(a,B)}R(\theta,Q_{\mathrm{G}_{\tau=\theta}}).
\end{align*}
Since the function 
\begin{eqnarray*}
	r_{i}(z):=\frac{1}{2}\log\left(\frac{1+z^{2}/v^{2}_{\varepsilon,\tilde{\varepsilon}}}{1+z^{2}/v^{2}_{\varepsilon}}\right)
	+\frac{1}{2}\frac{v^{2}_{\varepsilon,\tilde{\varepsilon}}+\theta_{i}^{2}}{v^{2}_{\varepsilon,\tilde{\varepsilon}}+z^{2}}
	-\frac{1}{2}\frac{v^{2}_{\varepsilon}+\theta_{i}^{2}}{v^{2}_{\varepsilon}+z^{2}}
\end{eqnarray*}
has a minimum at $z=\theta_{i}$,
for $\theta\in l_{2}$,
\begin{eqnarray}
	\mathop{\inf}_{\tau\in l_{2}} R(\theta,Q_{\mathrm{G}_{\tau}})
&=&
\left[
\mathop{\sum}_{i=1}^{\infty}
\frac{1}{2}\log\left(\frac{1+\theta^{2}_{i}/v^{2}_{\varepsilon,\tilde{\varepsilon}}}{1+\theta^{2}_{i}/v^{2}_{\varepsilon}}\right)
\right].
\label{Min_tau}
\end{eqnarray}
Since the minimax risk is greater than the maximin risk,
\begin{eqnarray}
	\mathop{\sup}_{\theta\in\Theta(a,B)}R(\theta,Q_{\mathrm{G}_{\tau=\tau^{*}(\varepsilon,\tilde{\varepsilon})}})
&\geq&
	\mathop{\inf}_{\tau\in l_{2}}\mathop{\sup}_{\theta\in\Theta(a,B)}R(\theta,Q_{\mathrm{G}_{\tau}})
\nonumber\\
&\geq&
\mathop{\sup}_{\theta\in\Theta(a,B)}\mathop{\inf}_{\tau\in l_{2}}R(\theta,Q_{\mathrm{G}_{\tau}})
\nonumber\\
&=&
\mathop{\sup}_{\theta\in\Theta(a,B)}R(\theta,Q_{\mathrm{G}_{\tau=\theta}}).
\label{minimaxmaximineq}
\end{eqnarray}

Finally,
we show 
that
\begin{align*}
	\mathop{\sup}_{\theta\in\Theta(a,B)}R(\theta,Q_{\mathrm{G}_{\tau=\theta}})
	\geq\mathop{\sup}_{\theta\in\Theta(a,B)}R(\theta,Q_{\mathrm{G}_{\tau=\tau^{*}(\varepsilon,\tilde{\varepsilon})}}).
\end{align*}
Substituting $\tau=\tau^{*}(\varepsilon,\tilde{\varepsilon})$ 
into (\ref{KL_Bayes_productNormal})
for any $\theta\in\Theta(a,B)$
yields
\begin{align*}
&\mathop{\sum}_{i=1}^{T(\varepsilon,\tilde{\varepsilon})}
	\left\{
	\frac{1}{2}\log\left(
		\frac{1+(\tau^{*}_{i}(\varepsilon,\tilde{\varepsilon}))^{2}/v^{2}_{\varepsilon,\tilde{\varepsilon}}}
		{1+(\tau^{*}_{i}(\varepsilon,\tilde{\varepsilon})^{2}/v^{2}_{\varepsilon})}
\right)\right\}
-R(\theta,Q_{\mathrm{G}_{\tau=\tau^{*}(\varepsilon,\tilde{\varepsilon})}})
\nonumber\\
&=\frac{1}{2}\mathop{\sum}_{i=1}^{\infty}
\left\{
\frac{v^{2}_{\varepsilon}+\theta^{2}_{i}}{v^{2}_{\varepsilon}+\left(\tau^{*}_{i}(\varepsilon,\tilde{\varepsilon})\right)^{2}}
-\frac{v^{2}_{\varepsilon,\tilde{\varepsilon}}+\theta^{2}_{i}}{v^{2}_{\varepsilon,\tilde{\varepsilon}}+\left(\tau^{*}_{i}(\varepsilon,\tilde{\varepsilon})\right)^{2}}
\right\}
\nonumber\\
&=\frac{1}{2}\mathop{\sum}_{i=1}^{\infty}
\left\{
\frac{1}{v^{2}_{\varepsilon}+\left(\tau^{*}_{i}(\varepsilon,\tilde{\varepsilon})\right)^{2}}
-\frac{1}{v^{2}_{\varepsilon,\tilde{\varepsilon}}+\left(\tau^{*}_{i}(\varepsilon,\tilde{\varepsilon})\right)^{2}}\right\}
\left(\theta^{2}_{i}-\left(\tau^{*}_{i}(\varepsilon,\tilde{\varepsilon})\right)^{2}\right)
\nonumber\\
&=\frac{1}{2}\mathop{\sum}_{i=1}^{\infty}
\frac{v^{2}_{\varepsilon,\tilde{\varepsilon}}-v^{2}_{\varepsilon}}{(v^{2}_{\varepsilon}+\left(\tau^{*}_{i}(\varepsilon,\tilde{\varepsilon})\right)^{2})(v^{2}_{\varepsilon,\tilde{\varepsilon}}+\left(\tau^{*}_{i}(\varepsilon,\tilde{\varepsilon})\right)^{2})}
(\theta^{2}_{i}-\left(\tau^{*}_{i}(\varepsilon,\tilde{\varepsilon})\right)^{2})
\nonumber\\
&=\mathop{\sum}_{i=1}^{\infty}\frac{1}{2}\{2\lambda(\varepsilon,\tilde{\varepsilon})a_{i}^{2}\}
\left\{\left(\tau^{*}_{i}(\varepsilon,\tilde{\varepsilon})\right)-\theta_{i}^{2}\right\}
\geq 0
\end{align*}
with equality if $\theta=\tau^{*}(\varepsilon,\tilde{\varepsilon})$.
Thus,
\begin{eqnarray*}
	\mathop{\sum}_{i=1}^{T(\varepsilon,\tilde{\varepsilon})}
	\left\{\frac{1}{2}\log\left(
			\frac{1+(\tau^{*}_{i}(\varepsilon,\tilde{\varepsilon}))^{2}/v^{2}_{\varepsilon,\tilde{\varepsilon}}}
			{1+(\tau^{*}_{i}(\varepsilon,\tilde{\varepsilon}))^{2}/v^{2}_{\varepsilon}}
	\right)\right\}
	=\mathop{\sup}_{\theta\in\Theta(a,B)}R(\theta,Q_{\mathrm{G}_{\tau=\tau^{*}(\varepsilon,\tilde{\varepsilon})}}).
\end{eqnarray*}
Since
\begin{eqnarray*}
	\mathop{\sup}_{\theta\in\Theta(a,B)}R(\theta,Q_{\mathrm{G}_{\tau=\theta}})
&\geq&
	R(\tau^{*}(\varepsilon,\tilde{\varepsilon}),Q_{\mathrm{G}_{\tau=\tau^{*}(\varepsilon,\tilde{\varepsilon})}})
\nonumber\\
&=&
\mathop{\sum}_{i=1}^{T(\varepsilon,\tilde{\varepsilon})}\left\{\frac{1}{2}
\log\left(
	\frac{1+(\tau^{*}_{i}(\varepsilon,\tilde{\varepsilon}))^{2}/v^{2}_{\varepsilon,\tilde{\varepsilon}}}
	{1+(\tau^{*}_{i}(\varepsilon,\tilde{\varepsilon}))^{2}/v^{2}_{\varepsilon}}
\right)\right\},
\end{eqnarray*}
it follows that
\begin{align*}
	\mathop{\sup}_{\theta\in\Theta(a,B)}R(\theta,Q_{\mathrm{G}_{\tau=\theta}})
	\geq\mathop{\sup}_{\theta\in\Theta(a,B)}R(\theta,Q_{\mathrm{G}_{\tau=\tau^{*}(\varepsilon,\tilde{\varepsilon})}}).
\end{align*}
\end{proof}

\vspace{4mm}

\begin{proof}[Proof of Lemma \ref{Lowerboundbysubproblem}]
First,
note that the following equivalence holds:
\begin{eqnarray*}
	Q_{\theta}\ll\widehat{Q}(\cdot;X) \text{ for all } \theta\in l_{2} \text{ if and only if } Q_{0}\ll\widehat{Q}(\cdot;X),
\end{eqnarray*}
for $P_{\theta}$-almost all $X$.
This is because two Gaussian measures $Q_{\theta}$ and $Q_{0}$ are mutually absolutely continuous if and only if $\theta\in l_{2}$.

Second,
from the above equivalence,
we have the following lower bound of the minimax risk:
for any $d\in\mathbb{N}$,
\begin{align*}
\mathop{\inf}_{\widehat{Q}\in\mathcal{D}}
&\mathop{\sup}_{\theta\in\Theta(a,B)}R(\theta,\widehat{Q})
\nonumber\\
&=
\mathop{\inf}_{\begin{subarray}{c}\widehat{Q}\in\mathcal{D}:\\ Q_{0}\ll\widehat{Q}(\cdot;\cdot)\end{subarray}}
	\mathop{\sup}_{\theta\in\Theta(a,B)}R(\theta,\widehat{Q})
\nonumber\\
&\geq
\mathop{\inf}_{\begin{subarray}{c}\widehat{Q}\in\mathcal{D}:\\ Q_{0}\ll\widehat{Q}(\cdot;\cdot)\end{subarray}}
\mathop{\sup}_{\begin{subarray}{c}\theta^{(d)}\in \Theta^{(d)}(a,B):\\ \theta_{d+i}=0\,\mathrm{for}\, i\in\mathbb{N}\end{subarray}}
R(\theta,\widehat{Q})
\nonumber\\
&=
\mathop{\inf}_{\begin{subarray}{c}\widehat{Q}\in\mathcal{D}:\\ Q_{0}\ll\widehat{Q}(\cdot;\cdot)\end{subarray}}
\mathop{\sup}_{\begin{subarray}{c}\theta^{(d)}\in \Theta^{(d)}(a,B):\\ \theta_{d+i}=0\,\mathrm{for}\, i\in\mathbb{N}\end{subarray}}
	\mathrm{E}_{X,Y|(\theta^{(d)},0,\ldots)}\left[\log\frac{\mathrm{d}Q_{(\theta^{(d)},0,\ldots)}}{\mathrm{d}\widehat{Q}(\cdot;X=x)}(y)\right],
\end{align*}
where we denote $\theta$ with $\theta_{i}=0$ for $i\geq d+1$ by $(\theta^{(d)},0,\ldots)$.

Third,
we further bound the previous lower bound.
Hereafter,
we fix $d\in\mathbb{N}$.
To do this,
we consider the decomposition of the density $\mathrm{d}Q_{(\theta^{(d)},0,\ldots)}/\mathrm{d}\widehat{Q}(\cdot;X=x)$ with respect to $\widehat{Q}(\cdot;X=x)$ as follows.
Let $\pi^{(d)}:\mathbb{R}^{\infty}\to\mathbb{R}^{d}$ be the projection $\pi^{(d)}(x)=(x_{1},\ldots,x_{d})$.
The projection $\pi^{(d)}$
induces a marginal probability measure $\widehat{Q}^{(d)}(\cdot;X=x)$ on $(\mathbb{R}^{d},\mathcal{R}^{d})$ and
a conditional probability measure $\widehat{Q}(\cdot|\pi^{(d)}(Y);X=x)$ on $(\mathbb{R}^{\infty},\mathcal{R}^{\infty})$ such that
for any measurable set $A\in \mathcal{R}^{\infty}$,
\begin{eqnarray*}
	\widehat{Q}(A;X=x)=\int \widehat{Q}(A|\pi^{(d)}(Y)=y^{(d)};X=x)\widehat{Q}^{(d)}(\mathrm{d}y^{(d)};X=x).
\end{eqnarray*}
In accordance with the decomposition of $\widehat{Q}(\cdot;X=x)$ 
into two probability measures $\widehat{Q}^{(d)}(\cdot;X=x)$ and $\widehat{Q}(\cdot|\pi^{(d)}(Y);X=x)$,
we decompose the density $\mathrm{d}Q_{(\theta^{(d)},0,\ldots)}/\mathrm{d}\widehat{Q}(\cdot;X=x)$ with respect to $\widehat{Q}(\cdot;X=x)$ as
\begin{eqnarray*}
\frac{\mathrm{d}Q_{(\theta^{(d)},0,\ldots)}}{\mathrm{d}\widehat{Q}(\cdot;X=x)}(y)
&=&\frac{\mathrm{d}Q^{(d)}_{\theta^{(d)}}}{\mathrm{d}\widehat{Q}^{(d)}(\cdot;X=x)}(y^{(d)})
\frac{\mathrm{d}Q_{(\theta^{(d)},0,\ldots)}(\cdot|\pi^{(d)}(Y)=y^{(d)})}{\mathrm{d}\widehat{Q}(\cdot|\pi^{(d)}(Y)=y^{(d)};X=x)}(y)
\end{eqnarray*}
for almost all $y$.
See p.119 in \citet{Pollard(2002)} for the decomposition.
From the decomposition of the density and from Jensen's inequality,
for $\theta^{(d)}\in\Theta^{(d)}(a,B)$ and for $\widehat{Q}\in\mathcal{D}$ such that $Q_{0}\ll \widehat{Q}(\cdot;X)$,
we have
\begin{align*}
&\mathrm{E}_{X,Y|(\theta^{(d)},0,\ldots)}\left[\log\frac{\mathrm{d}Q_{(\theta^{(d)},0,\ldots)}}{\mathrm{d}\widehat{Q}(\cdot;X=x)}(y)\right]
\nonumber\\
&=
\mathrm{E}_{X,Y|(\theta^{(d)},0,\ldots)}\left[
	\log\frac{\mathrm{d}Q^{(d)}_{\theta^{(d)}}}{\mathrm{d}\widehat{Q}^{(d)}(\cdot;X=x)}(y^{(d)})
+\log\frac{\mathrm{d}Q_{(\theta^{(d)},0,\ldots)}(\cdot|\pi^{(d)}(Y)=y^{(d)})}{\mathrm{d}\widehat{Q}(\cdot|\pi^{(d)}(Y)=y^{(d)};X=x)}
\right]
\nonumber\\
&\geq
\mathrm{E}_{X,Y|(\theta^{(d)},0,\ldots)}\left[
	\log\frac{\mathrm{d}Q^{(d)}_{\theta^{(d)}}}{\mathrm{d}\widehat{Q}^{(d)}(\cdot;X=x)}(y^{(d)})
\right].
\end{align*}
We denote the probability measure obtained by taking the conditional expectation of $\widehat{Q}^{(d)}(\cdot;X=x)$
conditioned by $\pi^{(d)}(X)$
under $P_{(\theta^{(d)},0,\ldots)}$
by $\widehat{Q}^{(d)} (\cdot ; \pi^{(d)} (X) = x^{(d)} )$,
because it does not depend on $\theta^{(d)}$ given $\pi^{(d)}(X)=x^{(d)}$.
By the definition of $\widehat{Q}^{(d)}(\cdot;\pi^{(d)}(X)=x^{(d)})$,
\begin{eqnarray*}
	1=\mathrm{E}_{X,Y|(\theta^{(d)},0,\ldots)}
	\left[\frac{\mathrm{d}\widehat{Q}^{(d)}(\cdot;X=x)}{\mathrm{d}\widehat{Q}^{(d)}(\cdot;\pi^{(d)}(X)=x^{(d)})}\bigg{|}\pi^{(d)}(X)=x^{(d)}\right]
	\text{ $P^{(d)}_{\theta^{(d)}}$-a.s.}.
\end{eqnarray*}
By Jensen's inequality and by the above equality,
\begin{align*}
	\mathrm{E}_{X,Y|(\theta^{(d)},0,\ldots)}&\left[-\log \frac{\mathrm{d}\widehat{Q}^{(d)}(\cdot;X=x)}{\mathrm{d}\widehat{Q}^{(d)}(\cdot;\pi^{(d)}(X)=x^{(d)})}\right]
	\nonumber\\
	&\geq
	-\log \mathrm{E}_{X,Y|(\theta^{(d)},0,\ldots)}\left[\frac{\mathrm{d}\widehat{Q}^{(d)}(\cdot;X=x)}{\mathrm{d}\widehat{Q}^{(d)}(\cdot;\pi^{(d)}(X)=x^{(d)})}\right]=0.
\end{align*}
Thus,
\begin{align*}
&\mathrm{E}_{X,Y|(\theta^{(d)},0,\ldots)}\left[
	\log\frac{\mathrm{d}Q^{(d)}_{\theta^{(d)}}}{\mathrm{d}\widehat{Q}^{(d)}(\cdot;X=x)}(y^{(d)})
\right]
\nonumber\\
&=
\mathrm{E}_{X,Y|(\theta^{(d)},0,\ldots)}\left[
	\log\frac{\mathrm{d}Q^{(d)}_{\theta^{(d)}}}{\mathrm{d}\widehat{Q}^{(d)}(\cdot;\pi^{(d)}(X)=x^{(d)})}(y^{(d)})
\right]
\nonumber\\
&\quad+
\mathrm{E}_{X,Y|(\theta^{(d)},0,\ldots)}\left[
	\log\frac{\mathrm{d}\widehat{Q}^{(d)}(\cdot;\pi^{(d)}(X)=x^{(d)})}{\mathrm{d}\widehat{Q}^{(d)}(\cdot;X=x)}(y^{(d)})
\right]
\nonumber\\
&\geq
\mathrm{E}_{\pi^{(d)}(X),\pi^{(d)}(Y)|\theta^{(d)}}\left[
	\log\frac{\mathrm{d}Q^{(d)}_{\theta^{(d)}}}{\mathrm{d}\widehat{Q}^{(d)}(\cdot;\pi^{(d)}(X)=x^{(d)})}(y^{(d)})
\right],
\end{align*}
where $\mathrm{E}_{\pi^{(d)}(X),\pi^{(n)}(Y)|\theta^{(d)}}$ is the expectation of $\pi^{(d)}(X)$ and $\pi^{(d)}(Y)$
with respect to $P^{(d)}_{\theta^{(d)}}$ and $Q^{(d)}_{\theta^{(d)}}$.
Hence,
\begin{align*}
	&\mathop{\inf}_{\widehat{Q}\in \mathcal{D}}\mathop{\sup}_{\theta\in\Theta(a,B)} R(\theta,\widehat{Q})
\nonumber\\
&\geq
\mathop{\inf}_{\widehat{Q}^{(d)}\in\mathcal{D}^{(d)}}
\mathop{\sup}_{\theta^{(d)}\in \Theta^{(d)}(a,B)}
\mathrm{E}_{\pi^{(d)}(X),\pi^{(d)}(Y)|\theta^{(d)}}\left[
	\log\frac{\mathrm{d}Q_{\theta^{(d)}}}{\mathrm{d}\widehat{Q}^{(d)}(\cdot;\pi^{(d)}(X)=x^{(d)})}
\right].
\end{align*}
\end{proof}

\section{Proofs of Lemmas in Section 4}\label{Appendix:ProofofSection4}

\begin{proof}[Proof of Lemma \ref{Oracle_blockwiseStein}]
From Lemma \ref{GaussianMeasures},
the Kullback--Leibler risk of $Q_{h^{(d)}_{\mathcal{B}(d)}}$ is given by
\begin{eqnarray*}
	R\left(\theta,Q_{h^{(d)}_{\mathcal{B}(d)}}\right)
	=\mathop{\sum}_{j=1}^{J} R_{b_{j}}\left(\theta_{B_{j}},Q^{(b_{j})}_{h^{(b_{j})}}\right)
	+\mathop{\sum}_{i=d+1}^{\infty}\frac{\theta_{i}^{2}}{\tilde{\varepsilon}^{2}},
\end{eqnarray*}
where for $j\in\{1,\ldots,J\}$,
$Q^{(b_{j})}_{h^{(b_{j})}}(\cdot|X_{B_{j}})$ is the Bayesian predictive distribution on $\mathbb{R}^{b_{j}}$
based on Stein's prior $h^{(b_{j})}(\theta_{B_{j}})$.
For the block $B_{j}$ with $b_{j}>2$, inequality (\ref{oracleineq}) holds.
For the block $B_{j}$ with $b_{j}\geq 2$,
the inequality
\begin{eqnarray*}
\frac{b_{j}}{2}\log\left(
\frac{1+(||\theta_{B_{j}}||^{2}/b_{j})/v^{2}_{\varepsilon,\tilde{\varepsilon}}}
{1+(||\theta_{B_{j}}||^{2}/b_{j})/v^{2}_{\varepsilon}}
\right)
\leq
\log\left(
\frac{v^{2}_{\varepsilon}}{v^{2}_{\varepsilon,\tilde{\varepsilon}}}
\right)
\end{eqnarray*}
holds
because the left hand side in the above inequality is monotone increasing with respect to $||\theta_{B_{j}}||$.
Thus,
\begin{eqnarray*}
	R(\theta,Q_{h^{(d)}_{\mathcal{B}(d)}})
\leq
J\log\left(
	\frac{v^{2}_{\varepsilon}}{v^{2}_{\varepsilon,\tilde{\varepsilon}}}
\right)
+\mathop{\sum}_{i=1}^{J}
\frac{b_{i}}{2}\log\left(
	\frac{1+(||\theta_{B_{i}}||^{2}/b_{i})/v^{2}_{\varepsilon,\tilde{\varepsilon}}}
	{1+(||\theta_{B_{i}}||^{2}/b_{i})/v^{2}_{\varepsilon}}
\right)
+\mathop{\sum}_{i=d+1}^{\infty}\frac{\theta^{2}_{i}}{\tilde{\varepsilon}^{2}}.
\end{eqnarray*}

From the same calculation as (\ref{Min_tau}) in Lemma \ref{LinearMinimax},
\begin{eqnarray*}
	\mathop{\inf}_{\widehat{Q}\in \mathcal{G}_{\mathrm{BW}}(\mathcal{B}(d))}
	R(\theta,\widehat{Q})=\mathop{\sum}_{i=1}^{J}
\frac{b_{i}}{2}\log\left(
\frac{1+(||\theta_{B_{i}}||^{2}/b_{i})/v^{2}_{\varepsilon,\tilde{\varepsilon}}}
	{1+(||\theta_{B_{i}}||^{2}/b_{i})/v^{2}_{\varepsilon}}
\right)
+\mathop{\sum}_{i=d+1}^{\infty}\frac{\theta^{2}_{i}}{\tilde{\varepsilon}^{2}}.
\end{eqnarray*}
Thus,
\begin{eqnarray}
	R\left(\theta,Q_{h^{(d)}_{\mathcal{B}(d)}}\right)
\leq
J\log\left(\frac{v^{2}_{\varepsilon}}{v^{2}_{\varepsilon,\tilde{\varepsilon}}}\right)
+\mathop{\inf}_{\widehat{Q}\in \mathcal{G}_{\mathrm{BW}}(\mathcal{B}(d))}R(\theta,\widehat{Q}).
\label{BSteinoracle}
\end{eqnarray}
Combining inequality (\ref{BSteinoracle}) with Lemma \ref{MimicMonotoneclass} yields (\ref{oracleineq_blockwiseStein}).
\end{proof}

\vspace{4mm}

\begin{proof}[Proof of Lemma \ref{MimicMonotoneclass}]
	It suffices to show that for any $\tau^{(d)}\in \mathcal{T}_{\mathrm{mon}}$,
	there exists $\bar{\tau}^{(d)}\in \mathcal{T}_{\mathrm{BW}}$
such that
\begin{eqnarray}
	R_{d}\left(\theta^{(d)},Q^{(d)}_{\mathrm{G}_{\bar{\tau}^{(d)}}}\right)
	\leq
	(1+\eta)R_{d}(\theta^{(d)},Q^{(d)}_{\mathrm{G}_{\tau^{(d)}}})
	+\frac{b_{1}}{2}\log\left(\frac{v^{2}_{\varepsilon}}{v^{2}_{\varepsilon,\tilde{\varepsilon}}}\right),
\label{KL_upper_G_bar}
\end{eqnarray}
where $Q^{(d)}_{\mathrm{G}_{\tau^{(d)}}}$ 
and $Q^{(d)}_{\mathrm{G}_{\bar{\tau}^{(d)}}}$
are Bayesian predictive distributions on $\mathbb{R}^{d}$
based on $\mathrm{G}_{\tau^{(d)}}:=\otimes_{i=1}^{d}\mathcal{N}(0,\tau_{i}^{2})$ and 
$\mathrm{G}_{\bar{\tau}^{(d)}}:=\otimes_{i=1}^{d}\mathcal{N}(0,\bar{\tau}_{i}^{2})$,
respectively.

For any $\tau^{(d)}\in \mathcal{T}_{\mathrm{mon}}$,
we define $\bar{\tau}^{(d)}=(\bar{\tau}_{1},\ldots,\bar{\tau}_{d})$ as
\begin{eqnarray}
\bar{\tau}_{i}=\left\{
	\begin{array}{l}
		\tau_{1} \text{ for } i\in B_{1}, \\
		\tau_{b_{1}+1} \text{ for } i\in B_{2}, \\
	\cdots \\
	        \tau_{b_{1}+\cdots+b_{J-1}+1} \text{ for } i\in B_{J}.
	\end{array}
	\right.
\end{eqnarray}
For $j\in\{1,\ldots,J\}$ and for $i\in B_{j}$,
let $\tau_{(j)}$ be $\bar{\tau}_{i}$
because $\bar{\tau}_{i}$ does not change in the same block $B_{j}$.
Then,
\begin{align}
	R_{d}&\left(\theta^{(d)},Q^{(d)}_{\mathrm{G}_{\bar{\tau}^{(d)}}}\right)
\nonumber\\
&=\mathop{\sum}_{i=1}^{d}\left[
\frac{1}{2}\log\left(\frac{1+\bar{\tau}_{i}^{2}/v^{2}_{\varepsilon,\tilde{\varepsilon}}}
{1+\bar{\tau}_{i}^{2}/v^{2}_{\varepsilon}}\right)
-\frac{1}{2}\frac{(v^{2}_{\varepsilon}-v^{2}_{\varepsilon,\tilde{\varepsilon}})\bar{\tau}^{2}_{i}}{(v^{2}_{\varepsilon,\tilde{\varepsilon}}+\bar{\tau}_{i}^{2})(v^{2}_{\varepsilon}+\bar{\tau}_{i}^{2})}
\right]
\nonumber\\
&\quad+
\mathop{\sum}_{i=1}^{d}\left[
\frac{1}{2}\frac{(v^{2}_{\varepsilon}-v^{2}_{\varepsilon,\tilde{\varepsilon}})\theta^{2}_{i}}{(v^{2}_{\varepsilon,\tilde{\varepsilon}}+\bar{\tau}_{i}^{2})(v^{2}_{\varepsilon}+\bar{\tau}_{i}^{2})}
\right].
\label{KL_mon}
\end{align}
From the inequality that $\tau_{i}\leq\bar{\tau}_{i}$ for $i\in\{1,\ldots,d\}$,
the second term in (\ref{KL_mon}) is bounded above as follows:
\begin{eqnarray*}
\mathop{\sum}_{i=1}^{d}\left[
\frac{1}{2}\frac{(v^{2}_{\varepsilon}-v^{2}_{\varepsilon,\tilde{\varepsilon}})\theta^{2}_{i}}{(v^{2}_{\varepsilon,\tilde{\varepsilon}}+\bar{\tau}_{i}^{2})(v^{2}_{\varepsilon}+\bar{\tau}_{i}^{2})}
\right]
&\leq&
\mathop{\sum}_{i=1}^{d}\left[
	\frac{1}{2}\frac{(v^{2}_{\varepsilon}-v^{2}_{\varepsilon,\tilde{\varepsilon}})\theta^{2}_{i}}{(v^{2}_{\varepsilon,\tilde{\varepsilon}}+\tau_{i}^{2})(v^{2}_{\varepsilon}+\tau_{i}^{2})}
\right].
\end{eqnarray*}
Thus,
\begin{align}
	R_{d}&\left(\theta^{(d)},Q^{(d)}_{\mathrm{G}_{\bar{\tau}^{(d)}}}\right)
\nonumber\\
&\leq
\mathop{\sum}_{i=1}^{d}\left[
	\frac{1}{2}\log\left(\frac{1+\bar{\tau}_{i}^{2}/v^{2}_{\varepsilon,\tilde{\varepsilon}}}{1+\bar{\tau}_{i}^{2}/v^{2}_{\varepsilon}}\right)
-\frac{1}{2}\frac{(v^{2}_{\varepsilon}-v^{2}_{\varepsilon,\tilde{\varepsilon}})\bar{\tau}^{2}_{i}}{(v^{2}_{\varepsilon,\tilde{\varepsilon}}+\bar{\tau}_{i}^{2})(v^{2}_{\varepsilon}+\bar{\tau}_{i}^{2})}\right]
\nonumber\\
&\quad+
\mathop{\sum}_{i=1}^{d}\left[
	\frac{1}{2}\frac{(v^{2}_{\varepsilon}-v^{2}_{\varepsilon,\tilde{\varepsilon}})\theta^{2}_{i}}{(v^{2}_{\varepsilon,\tilde{\varepsilon}}+\tau_{i}^{2})(v^{2}_{\varepsilon}+\tau_{i}^{2})}
\right].
\label{KL_mon_upper}
\end{align}
For the first term on the right hand side of (\ref{KL_mon_upper}),
from the definition of $\bar{\tau}^{(d)}$,
\begin{align}
	\mathop{\sum}_{i=1}^{d}&\left[
	\frac{1}{2}\log\left(\frac{1+\bar{\tau}_{i}^{2}/v^{2}_{\varepsilon,\tilde{\varepsilon}}}{1+\bar{\tau}_{i}^{2}/v^{2}_{\varepsilon}}\right)
-\frac{1}{2}\frac{(v^{2}_{\varepsilon}-v^{2}_{\varepsilon,\tilde{\varepsilon}})\bar{\tau}^{2}_{i}}{(v^{2}_{\varepsilon,\tilde{\varepsilon}}+\bar{\tau}_{i}^{2})(v^{2}_{\varepsilon}+\bar{\tau}_{i}^{2})}\right]
\nonumber\\
&=\mathop{\sum}_{j=1}^{J}\mathop{\sum}_{m\in B_{j}}\left[
\frac{1}{2}\log\left(\frac{1+\tau_{(j)}^{2}/v^{2}_{\varepsilon,\tilde{\varepsilon}}}{1+\tau_{(j)}^{2}/v^{2}_{\varepsilon}}\right)
-\frac{1}{2}\frac{(v^{2}_{\varepsilon}-v^{2}_{\varepsilon,\tilde{\varepsilon}})\tau^{2}_{(j)}}{(v^{2}_{\varepsilon,\tilde{\varepsilon}}+\tau_{(j)}^{2})(v^{2}_{\varepsilon}+\tau_{(j)}^{2})}\right].
\label{KL_mon_upper_equiv}
\end{align}
Note that the function $f(x)$ defined by
\begin{eqnarray*}
f(x)=\log\left(\frac{1+x/v^{2}_{\varepsilon,\tilde{\varepsilon}}}{1+x/v^{2}_{\varepsilon}}\right)
-\frac{(v^{2}_{\varepsilon}-v^{2}_{\varepsilon,\tilde{\varepsilon}})x}{(v^{2}_{\varepsilon}+x)(v^{2}_{\varepsilon,\tilde{\varepsilon}}+x)}
\end{eqnarray*}
is monotone increasing with respect to $x\in [0,\infty)$
and then $f(x)\leq \log(v^{2}_{\varepsilon}/v^{2}_{\varepsilon,\tilde{\varepsilon}})$
because 
\begin{eqnarray*}
f'(x)=(v^{2}_{\varepsilon}-v^{2}_{\varepsilon,\tilde{\varepsilon}})(v^{2}_{\varepsilon}+v^{2}_{\varepsilon,\tilde{\varepsilon}})
\frac{x}{(x+v^{2}_{\varepsilon})(x+v^{2}_{\varepsilon,\tilde{\varepsilon}})}\geq 0.
\end{eqnarray*}
Since $f(x)\leq \log(v^{2}_{\varepsilon}/v^{2}_{\varepsilon,\tilde{\varepsilon}})$ for $x\in[0,\infty)$,
\begin{eqnarray*}
\mathop{\sum}_{m\in B_{1}}\left[
\frac{1}{2}\log\left(
\frac{1+\tau_{(1)}^{2}/v^{2}_{\varepsilon,\tilde{\varepsilon}}}{1+\tau_{(1)}^{2}/v^{2}_{\varepsilon}}\right)
-\frac{1}{2}\frac{(v^{2}_{\varepsilon}-v^{2}_{\varepsilon,\tilde{\varepsilon}})\tau^{2}_{(1)}}{(v^{2}_{\varepsilon,\tilde{\varepsilon}}+\tau_{(1)}^{2})(v^{2}_{\varepsilon}+\tau_{(1)}^{2})}\right]
&\leq&
\frac{b_{1}}{2}\log \left(\frac{v^{2}_{\varepsilon}}{v^{2}_{\varepsilon,\tilde{\varepsilon}}}\right).
\end{eqnarray*}
Thus,
the first term in (\ref{KL_mon_upper_equiv}) is bounded as follows:
\begin{align}
	\mathop{\sum}_{j=1}^{J}&\mathop{\sum}_{m\in B_{j}}\left[
\frac{1}{2}\log\left(
\frac{1+\tau_{(j)}^{2}/v^{2}_{\varepsilon,\tilde{\varepsilon}}}{1+\tau_{(j)}^{2}/v^{2}_{\varepsilon}}\right)
-\frac{1}{2}\frac{(v^{2}_{\varepsilon}-v^{2}_{\varepsilon,\tilde{\varepsilon}})\tau^{2}_{(j)}}{(v^{2}_{\varepsilon,\tilde{\varepsilon}}+\tau_{(j)}^{2})(v^{2}_{\varepsilon}+\tau_{(j)}^{2})}\right]
\nonumber\\
&\leq
\frac{b_{1}}{2}\log\left(\frac{v^{2}_{\varepsilon}}{v^{2}_{\varepsilon,\tilde{\varepsilon}}}\right)
\nonumber\\
&\quad+
\mathop{\sum}_{j=2}^{J}\mathop{\sum}_{m\in B_{j}}\left[
	\frac{1}{2}\log\left(
	\frac{1+\tau_{(j)}^{2}/v^{2}_{\varepsilon,\tilde{\varepsilon}}}{1+\tau_{(j)}^{2}/v^{2}_{\varepsilon}}\right)
-\frac{1}{2}\frac{(v^{2}_{\varepsilon}-v^{2}_{\varepsilon,\tilde{\varepsilon}})\tau^{2}_{(j)}}{(v^{2}_{\varepsilon,\tilde{\varepsilon}}+\tau_{(j)}^{2})(v^{2}_{\varepsilon}+\tau_{(j)}^{2})}\right]
\nonumber\\
&=
\frac{b_{1}}{2}\log\left(\frac{v^{2}_{\varepsilon}}{v^{2}_{\varepsilon,\tilde{\varepsilon}}}\right)
\nonumber\\
&\quad+
\mathop{\sum}_{j=2}^{J}B_{j}\left[
\frac{1}{2}\log\left(\frac{1+\tau_{(j)}^{2}/v^{2}_{\varepsilon,\tilde{\varepsilon}}}{1+\tau_{(j)}^{2}/v^{2}_{\varepsilon}}\right)
-\frac{1}{2}\frac{(v^{2}_{\varepsilon}-v^{2}_{\varepsilon,\tilde{\varepsilon}})\tau^{2}_{(j)}}{(v^{2}_{\varepsilon,\tilde{\varepsilon}}+\tau_{(j)}^{2})(v^{2}_{\varepsilon}+\tau_{(j)}^{2})}\right].
\label{KL_first_mon_upper}
\end{align}
From the assumption on the cardinalities of the blocks
and
from the inequality that 
$f(\tau_{(j)})\leq f(\tau_{i})$ for $i\in B_{j-1}$,
\begin{align}
\mathop{\sum}_{j=2}^{J}&B_{j}\left[
\frac{1}{2}\log\left(\frac{1+\tau_{(j)}^{2}/v^{2}_{\varepsilon,\tilde{\varepsilon}}}{1+\tau_{(j)}^{2}/v^{2}_{\varepsilon}}\right)
-\frac{1}{2}\frac{(v^{2}_{\varepsilon}-v^{2}_{\varepsilon,\tilde{\varepsilon}})\tau^{2}_{(j)}}{(v^{2}_{\varepsilon,\tilde{\varepsilon}}+\tau_{(j)}^{2})(v^{2}_{\varepsilon}+\tau_{(j)}^{2})}\right]
\nonumber\\
&\leq
(1+\eta)\mathop{\sum}_{j=2}^{J}B_{j-1}\left[
\frac{1}{2}\log\left(\frac{1+\tau_{(j)}^{2}/v^{2}_{\varepsilon,\tilde{\varepsilon}}}{1+\tau_{(j)}^{2}/v^{2}_{\varepsilon}}\right)
-\frac{1}{2}\frac{(v^{2}_{\varepsilon}-v^{2}_{\varepsilon,\tilde{\varepsilon}})\tau^{2}_{(j)}}{(v^{2}_{\varepsilon,\tilde{\varepsilon}}+\tau_{(j)}^{2})(v^{2}_{\varepsilon}+\tau_{(j)}^{2})}\right]
\nonumber\\
&\leq
(1+\eta)\mathop{\sum}_{j=2}^{J}\mathop{\sum}_{m\in B_{j-1}}\left[
	\frac{1}{2}\log\left(\frac{1+\tau_{m}^{2}/v^{2}_{\varepsilon,\tilde{\varepsilon}}}{1+\tau_{m}^{2}/v^{2}_{\varepsilon}}\right)
-\frac{1}{2}\frac{(v^{2}_{\varepsilon}-v^{2}_{\varepsilon,\tilde{\varepsilon}})\tau^{2}_{m}}{(v^{2}_{\varepsilon,\tilde{\varepsilon}}+\tau_{m}^{2})(v^{2}_{\varepsilon}+\tau_{m}^{2})}\right]
\nonumber\\
&=
(1+\eta)\mathop{\sum}_{i=1}^{d}\left[
	\frac{1}{2}\log\left(\frac{1+\tau_{i}^{2}/v^{2}_{\varepsilon,\tilde{\varepsilon}}}{1+\tau_{i}^{2}/v^{2}_{\varepsilon}}\right)
-\frac{1}{2}\frac{(v^{2}_{\varepsilon}-v^{2}_{\varepsilon,\tilde{\varepsilon}})\tau^{2}_{i}}{(v^{2}_{\varepsilon,\tilde{\varepsilon}}+\tau_{i}^{2})(v^{2}_{\varepsilon}+\tau_{i}^{2})}\right].
\label{KL_first_mon_upper_upper}
\end{align}
From (\ref{KL_mon_upper}), (\ref{KL_first_mon_upper}), and (\ref{KL_first_mon_upper_upper}),
we obtain (\ref{KL_upper_G_bar}).
\end{proof}

\bibliographystyle{imsart-nameyear}
\bibliography{Yano_Pinsker}
\end{document}